\documentclass[12pt,reqno, oneside]{amsart}
\makeatletter
\@namedef{subjclassname@2020}{%
  \textup{2020} Mathematics Subject Classification}
\makeatother
\usepackage{amsmath,amssymb,latexsym,esint,cite,mathrsfs,slashed, dsfont}
\usepackage{verbatim,cite,relsize,color,soul}
\usepackage[latin1]{inputenc}
\usepackage{microtype}
\usepackage{color,enumitem,graphicx, xcolor}
\usepackage{mathtools}
\mathtoolsset{showonlyrefs}
\usepackage{bm} 
\usepackage[colorlinks=true,urlcolor=blue, citecolor=red,linkcolor=blue,
linktocpage,pdfpagelabels, bookmarksnumbered,bookmarksopen]{hyperref}
\usepackage[hyperpageref]{backref}
\usepackage[english]{babel}
\usepackage{geometry}
\usepackage{soul}

\numberwithin{equation}{section}
\newtheorem{theorem}{Theorem}[section]
\newtheorem{lemma}[theorem]{Lemma}

\newtheorem{proposition}[theorem]{Proposition}

\theoremstyle{definition}
\newtheorem{definition}[theorem]{Definition}
\newtheorem{remark}[theorem]{Remark}
\newtheorem*{ackno}{Acknowledgements}

\newcommand{\R}{\mathbb R}

\def\({\left(}
\def\){\right)}
\def\<{\left\langle}
\def\>{\right\rangle}

\DeclareMathOperator*{\GN}{GN}
\DeclareMathOperator*{\Sob}{Sob}
\DeclareMathOperator*{\rad}{rad}

\title[Energy critical NLS system with three waves interaction]
{Three wave interaction solitons for an energy critical Schr\"odinger system}

\author[L. Forcella, X. Luo, X. Yang
]{Luigi Forcella, Xiao Luo, and Xiaolong Yang}

\address[L. Forcella]{Department of Mathematics, University of Pisa, Largo Bruno Pontecorvo, 5, 56127, Pisa, Italy}
\email{luigi.forcella@unipi.it}

\address[X. Luo]{School of Mathematics, Hefei University of Technology, Hefei, 230009, P. R. China}
\email{luoxiao@hfut.edu.cn}

\address[X. Yang]{Xiaolong Yang, School of Mathematics and Statistics, Henan University, Kaifeng, 475004, P. R. China} 
\email{xlyang@henu.edu.cn}



\begin{document}
	
\begin{abstract}
 We investigate standing waves for the energy critical Schr\"odinger system with three waves interaction arising as a model for the Raman amplification in a plasma. Several results are proved: simultaneous existence of stable and unstable standing waves, existence of global solutions, and absence of small data scattering. Our main results show some specific features arising from the three waves interaction differently from the classical energy critical Schr\"odinger equation, and they support some experimental observations on  Raman amplification.
\end{abstract}

\maketitle
	
\section{Introduction}

In this paper, we consider the following nonlinear Schr\"odinger system with three waves interaction

\begin{equation}\label{eqA0.1}
\begin{cases}
i \partial_{t} \psi_{1}=-\Delta \psi_{1}-\left|\psi_{1}\right|^{4} \psi_{1}-\alpha \psi_{3} \overline{\psi}_{2}, \\
i \partial_{t} \psi_{2}=-\Delta \psi_{2}-\left|\psi_{2}\right|^{4} \psi_{2}-\alpha \psi_{3} \overline{\psi}_{1}, \\
i \partial_{t} \psi_{3}=-\Delta \psi_{3}-\left|\psi_{3}\right|^{4} \psi_{3}-\alpha \psi_{1} \psi_{2}.
\end{cases}
\end{equation}
Here $\psi_j=\psi_j(t,x)$ with $j\in\{1,2,3\}$, are complex-valued functions $\psi_j:\R\times \R^3\mapsto \mathbb C$, with $\overline \psi_j$ denoting the complex conjugate, $\alpha$ is a positive real parameter. From now on we use the compact notation
\[
{\bm \psi}(t,x)=\left(\psi_1(t,x),\psi_2(t,x),\psi_3(t,x)\right).
\]
\vskip1mm

System \eqref{eqA0.1} models the interaction in a plasma between an incident laser field, a backscattered Raman field, and an electronic wave and is related to the Raman amplification in a plasma \cite{RDR}. Roughly speaking, the Raman amplification is an instability phenomenon taking place when an incident laser field propagates into a plasma (see \cite{HA} and the introduction in \cite{PA}). As explained in \cite{PA},  the laser  field, entering a plasma,  is backscattered by a Raman type process and the interaction of the two waves generates an electronic plasma wave. Then the three waves together produce a change in the ions' density which in turn affects the waves. The exact derivation of \eqref{eqA0.1} from the above physical picture can be found in \cite{CCO,FLY}.

\vskip1mm

The system can also be written in Hamiltonian form
\begin{equation*}
\partial_t {\bm \psi}(t,x)=-i E'({\bm \psi}(t,x)),
\end{equation*}
where the following quantities are conserved along the flow: the energy, defined by
\begin{equation}\label{eq:energy}
E({\bm \psi}(t))=\sum^{3}_{j=1}\left(\frac{1}{2}\|\nabla \psi_{j}(t)\|^2_{L^2(\R^3)}-\frac{1}{6}\|\psi_j(t)\|^6_{L^6(\R^3)}  \right)-\alpha \mathrm{Re}\int_{\R^3} \left(\psi_1\psi_2\overline{\psi}_3\right)(t) dx,
\end{equation}
and the mixed masses
\begin{equation}\label{eq:cons-masses}
\begin{aligned}
&Q_1({\bm \psi}(t))=\|\psi_1(t)\|^2_{L^2(\R^3)}+\|\psi_3(t)\|^2_{L^2(\R^3)}, \\  &
Q_2({\bm \psi}(t))=\|\psi_2(t)\|^2_{L^2(\R^3)}+\|\psi_3(t)\|^2_{L^2(\R^3)}.
\end{aligned}
\end{equation}

\noindent As usual, conservation means that the previous quantities are not dependent on time, or alternatively $E({\bm \psi}(t))=E({\bm \psi}(0))$,  $Q_1({\bm \psi}(t))=Q_1({\bm \psi}(0))$, and $Q_2({\bm \psi}(t))=Q_2({\bm \psi}(0))$ for any time $t$ in the maximal interval of existence $[0,T_{\max})$ (without loss of generality, we only consider positive times). 

\vskip1mm

A standing wave for \eqref{eqA0.1} is a solution of the form  
\[
{\bm \psi}(t)=(e^{i\lambda_1t}u_1(x), e^{i\lambda_2t}u_2(x), e^{i\lambda_3t}u_3(x)),
\]
where $\lambda_1,\lambda_2,\lambda_3$ are real numbers and ${\bf u}=(u_1,u_2,u_3)\in H^1(\R^3,\mathbb{C}^3)$ satisfies the system of elliptic equations
\begin{equation}\label{eqA0.2}
\begin{cases}
-\Delta u_{1}+ \lambda_1u_{1}=\left|u_{1}\right|^{4} u_{1}+\alpha u_{3} \overline{u}_{2}, \\
-\Delta u_{2}+ \lambda_2u_{2}=\left|u_{2}\right|^{4} u_{2}+\alpha u_{3} \overline{u}_{1}, \\
-\Delta u_{3}+ \lambda_3u_{3}=\left|u_{3}\right|^{4} u_{3}+\alpha u_{1} u_{2},
\end{cases}
\end{equation}
where $\lambda_3=\lambda_1+\lambda_2$. If the energy critical power-type nonlinearities are replaced by sub-critical ones, namely if one considers 
\begin{equation}\label{eqA0.2sub-crit}
\begin{cases}
-\Delta u_{1}+ \lambda_1u_{1}=\left|u_{1}\right|^{p-2} u_{1}+\alpha u_{3} \overline{u}_{2}, \\
-\Delta u_{2}+ \lambda_2u_{2}=\left|u_{2}\right|^{p-2} u_{2}+\alpha u_{3} \overline{u}_{1}, \\
-\Delta u_{3}+ \lambda_3u_{3}=\left|u_{3}\right|^{p-2} u_{3}+\alpha u_{1} u_{2},
\end{cases}
\end{equation}
with $u_j:\mathbb R^N\to \mathbb C$, $j\in\{1,2,3\}$, and $2<p<2^*$ ($2^*=+\infty$ if $N\leq2$, $2^*=\frac{2N}{N-2}$ if $N\ge3$) for fixed frequency $\lambda_j \in \R$, under certain conditions, the existence, uniqueness and multiplicity of solutions of \eqref{eqA0.2} have been studied in \cite{PA,WJ,WS}. It is worth mentioning that in the present paper we consider the 3D physical case.  
From a physical point of view, it is shown in \cite{IE} that three wave interaction standing waves can be used for optical switching in a polarization-gate geometry. 

\vskip2mm

Since the mixed masses are preserved quantities along the evolution, one can get solutions to \eqref{eqA0.2} by looking for critical points of the energy functional $E({\bf u})$ constrained on
\begin{equation}
S(a_1,a_2):=\left\{{\bf u}\in H^1(\R^3,\mathbb{C}^3)\quad \hbox{ s.t. }\quad Q_1({\bf u})=a^2_1, \quad  Q_2({\bf u})=a^2_2\right\},
\end{equation}
where $a_1,a_2>0$ are prescribed positive parameters. Then, $\lambda_j \in \R$  in \eqref{eqA0.2}, $j\in\{1,2,3\}$, appear as  Lagrange multipliers with respect to the mass constraints. Existence of global minimizers for $E({\bf u})$ on $S(a_1,a_2)$ are proved in \cite{AH,KO}. In addition, for the dynamical properties of related standing waves for \eqref{eqA0.1}, see \cite{LO,CCO,CCO1,MM,AH,CC} and the references therein.
\vskip2mm

 We continue the study initiated in \cite{FLY} on the system in subcritical regime, i.e., \eqref{eqA0.2sub-crit}, and in this paper we focus on the more challenging problem of standing waves for the energy critical system \eqref{eqA0.2}.
 Note that the coupling terms are of mass-subcritical type and sign-indefinite, then we are dealing with a special mass-mixed case, which is more complicated.
\vskip1mm

We shall focus on physical states with different energy levels. We start by giving a few definitions.
\begin{definition}
We say that ${\bf u}_0$ is a ground state of \eqref{eqA0.2} on $S(a_1,a_2)$ provided  $dE \vert_{S(a_1,a_2)}({\bf u}_0)=0$ and
\begin{equation*}
E({\bf u}_0)=\inf \left\{E({\bf u}) \hbox{ s.t. } dE\vert_{S(a_1,a_2)}(u)=0\ \hbox{and}\ {\bf u}\in S(a_1,a_2)\right\}.
\end{equation*}
We say that ${\bf v}_0$ is an excited state of \eqref{eqA0.2} on $S(a_1,a_2)$ if $dE \vert_{S(a_1,a_2)}({\bf v}_0)=0$ and
\begin{equation*}
 E({\bf v}_0)>\inf \left\{E({\bf u}) \hbox{ s.t. } dE\vert_{S(a_1,a_2)}(u)=0\ \text{and}\ {\bf u}\in S(a_1,a_2)\right\}.
\end{equation*}
The set of ground states will be denoted by $\mathcal G$.
\end{definition}

To study mass-synchronised asymptotic of ground states and excited states, we introduce the following minimization problem
\begin{equation}\label{def:m0}
m_0(a_1,a_2):=\inf_{{\bf u}\in S(a_1,a_2)}E_0({\bf u}),
\end{equation}
where
\begin{equation}\label{energy-limit}
E_0({\bf u}):=\frac{1}{2}\sum^{3}_{j=1}\|\nabla u_j\|^2_{L^2(\R^3)}-\alpha\mathrm{Re}\int_{\R^3} u_1u_2\overline{u}_3 dx.
\end{equation}
It is well known that
\begin{equation*}
C_{\Sob}:=\inf_{u\in \dot H^{1}(\R^3)\setminus\{0\}}\frac{\|\nabla u\|^2_{L^2(\mathbb R^3)}}{\|u\|^2_{L^6(\mathbb R^3)}}
\end{equation*}
is achieved by the family of functions 
\begin{equation}\label{Sobolev}
U_\varepsilon(x)=\frac{3^{\frac{1}{4}}\varepsilon^{\frac{1}{2}}}{(\varepsilon^2+|x|^2)^{\frac{1}{2}}},\quad \varepsilon>0.
\end{equation}
We introduce the following constant $D$, which will play a crucial role in the rest of the paper:
\begin{equation}\label{def:D}
D:=\frac{2^{\frac{2}{3}}C_{\Sob}^{\frac{1}{4}}\|W\|^{\frac{2}{3}}_{L^2(\mathbb R^3)}}{3^{\frac{11}{12}}\alpha^{\frac{2}{3}}},
\end{equation}
where $W$ is the unique positive radial solution of \begin{equation}\label{eq:w}
-\Delta W+W-W^2=0 \quad \hbox{ in }\quad  \R^3.
\end{equation}

\vskip2mm

We can now state our main results.
\begin{theorem}\label{th1.1}
Let $\alpha, a_1, a_2>0$. Suppose that $\max\{a_1,a_2\}<D$, with $D$ as in \eqref{def:D}. Then we have:\smallskip

\noindent \textup{(i)} there exist a ground state ${\bf u}_{a_1,a_2}$ and an excite state ${\bf v}_{a_1,a_2}$ of \eqref{eqA0.2} on $S(a_1,a_2)$;\smallskip

\noindent \textup{(ii)} as $(a_1,a_2)\to (0^+,0^+)$, $\|\nabla {\bf u}_{a_1,a_2}\|_{L^2(\mathbb R^3)}\to 0$ and $E({\bf u}_{a_1,a_2})\to 0$; furthermore, 
$\|\nabla {\bf v}_{a_1,a_2}\|_{L^2(\mathbb R^3)} \to C_{\Sob}^{\frac{3}{2}}$ and $E({\bf v}_{a_1,a_2})\to \frac{1}{3}C_{\Sob}^{\frac{3}{2}}$;\smallskip

\noindent \textup{(iii)} fix $\alpha > 0$, assume $\epsilon:=a_1=a_2 \to 0^+$, then we have
\begin{equation*}
\|\nabla {\bf u}_{\epsilon,\epsilon}\|_{L^2(\mathbb R^3)}^{2}\sim \epsilon^6,\qquad E({\bf u}_{\epsilon,\epsilon})\sim \epsilon^6,
\end{equation*}
and furthermore, for every sequence  $\epsilon_n\to 0^+$, the rescaled family
\[w_n(x):=\epsilon_n^{-4}{\bf u}_{\epsilon_n,\epsilon_n}(\epsilon_n^{-2}x)\]
has a subsequence converging in $H^1(\R^3,\mathbb{C}^3)$ to a minimizer for $m_0(1,1)$ (see the definition \eqref{def:m0}).
\end{theorem}

We list a few comments that the results above deserve.
We list a few relevant comments that the results above deserve.
\begin{remark}
Theorem \ref{th1.1} \textup{(i)} indicates that there are two physical states of \eqref{eqA0.2}. The ground state is at a negative energy level, and the excited state lies at a positive energy level. The set of ground states $\mathcal G$, containing a-priori complex-valued ground states, has the following structure:
\begin{equation*}
\mathcal G=\left\{ (e^{i\theta_1}u_1,e^{i\theta_2}u_2,e^{i(\theta_1+\theta_2)}u_3) \quad \hbox{ s.t. } \quad \theta_1,\theta_2\in \R \right\},
\end{equation*}
where $(u_1,u_2,u_3)\in S(a_1,a_2)$ is a positive, radial ground state of \eqref{eqA0.2}. See the proof of Theorem \ref{th1.1} for details.\smallskip

\noindent Theorem \ref{th1.1} \textup{(ii)} shows a mass collapse profile of the two kind of physical states, one tends to vanish and the other may tend to the Aubin-Talenti bubble. 
In the context of normalized solutions for the Schr\"{o}dinger equation with critical Sobolev exponent, such asymptotic phenomenon (depending on other parameters instead of the mixed masses) has been recently observed in \cite{Soave2}, and depicted more clearly in \cite{WW}.\smallskip 

\noindent  Furthermore, by precisely calculating the upper bound of the ground state energy, we provide in Theorem \ref{th1.1} \textup{(iii)} a precise refined mass collapse profile of ground states. This is quite new in the literature and motivated by the papers \cite{gl,GL} concerning ground states of two-component attractive Bose-Einstein condensates. Note from \eqref{def:D} that $a_1, a_2$ can be taken arbitrary large by taking $\alpha$ small enough and the above mass collapse phenomenon occurs as the coupling frequency $\alpha$ tends to infinity.
\end{remark}

\begin{remark}
Compared to the energy subcritical case considered in our previous work \cite[Theorem 2]{FLY}, a lower bound of the coupling frequency $\alpha$ is removed in searching for excited states. The main reason is that we here improve the control of the energy level as follows:
\begin{equation}\label{moun-pass}
m^-(a_1,a_2)<\frac{1}{3}C_{\Sob}^{\frac{3}{2}}+m^+(a_1,a_2);
\end{equation}
namely, the excited state energy level is less than the usual critical threshold plus the ground state energy, see Lemma \ref{lem3.3} for details. The choice of the testing paths used to prove \eqref{moun-pass} is motivated by \cite{WW}. Note that, due to the introduction of three waves interaction term which is sign-indefinite, the treatment of this energy estimate requires new ideas and more refined treatment. See the  recent paper \cite{LWYZ} by the second and third author  for a similar sign-indefinite variational problem. 
\end{remark}

We now introduce the dynamical results on standing waves contained in the paper. In order to state them, we start by recalling the following definitions.

\begin{definition} \rm
\textup{(i)} We say that the set $\mathcal G$ is orbitally stable if $\mathcal G\neq \emptyset$ and for any $\varepsilon>0$, there exists a $\delta>0$ such that,  provided that the initial datum ${\bm \psi}_0=\left(\psi_1(0),\psi_2(0),\psi_3(0)\right)$ for \eqref{eqA0.1} satisfies
\begin{equation*}
\inf_{{\bf u}\in \mathcal G}\|{\bm \psi}_0-{\bf u}\|_{H^1(\R^3, \mathbb C^3)}< \delta,
\end{equation*}
then $\bm{\psi}(t)$ is globally defined and
\begin{equation*}\sup_{t\in \R}
\inf_{{\bf u}\in \mathcal G}\|{\bm \psi}(t)-{\bf u}\|_{H^1(\R^3, \mathbb C^3)}< \varepsilon,
\end{equation*}
where ${\bm \psi}(t)$ is the solution to \eqref{eqA0.1} corresponding to the initial condition ${\bm \psi}_0$.\smallskip

\noindent \textup{(ii)} A standing wave $(e^{i\lambda_1t}u_1,e^{i\lambda_2t}u_2,e^{i\lambda_3t}u_3)$ is said to be strongly unstable if, for any $\varepsilon>0$, there exists ${\bm \psi}_0\in H^1(\R^3,\mathbb{C}^3)$ such that $\|{\bf u}-{\bm \psi}_0\|_{H^1(\R^3, \mathbb C^3)}<\varepsilon$, and ${\bm \psi}(t)$ blows-up in finite time. 
\end{definition}

Note that the orbital stability of the set $\mathcal G$ implies the global existence of solutions to \eqref{eqA0.1} for initial datum close enough to the set $\mathcal G$. We underline that this fact is nontrivial due to that energy critical exponent appearing  in \eqref{eqA0.1}.
In the energy subcritical range $2<p<6$, such an orbital stability results for ground states related to  \eqref{eqA0.2sub-crit}  has been proved in \cite{AH} in the mass sub-critical range $2<p<\frac{10}{3}$, and in   our previous work \cite{FLY} for the mass critical/super-critical  and energy  subcritical range $\frac{10}{3}\leq p<6$. \\

We now extend it the energy critical case.

\begin{theorem}\label{th1.2}
Let $\alpha, a_1, a_2>0$ and $\max\{a_1,a_2\}<D$, where $D$ is defined in \eqref{def:D}. Then we have: \smallskip
 
\noindent \textup{(i)}
the ground state set $\mathcal G$ is orbitally stable;\smallskip

\noindent \textup{(ii)} the standing wave constructed by  ${\bm \psi}(t,x)=\left(e^{i\lambda_1 t}v_1,e^{i\lambda_1 t}v_2,e^{i(\lambda_1+\lambda_2)t}v_3\right)$ with the excited state ${\bf v}_{a_1,a_2}$ is strongly unstable.
\end{theorem}

Theorem \ref{th1.2} gives the simultaneous existence of stable and unstable standing waves for \eqref{eqA0.1}. The proof of   point \textup{(i)} of Theorem \ref{th1.2} follows the approach proposed in \cite{CaW}, which was  recently used to deal with the nonlinear Schr\"{o}dinger equations with critical Sobolev exponent. To obtain the orbital stability of the ground states set, two elements are essential, see the recent advances \cite{Soave2,JJL} regarding critical NLS equations with mixed nonlinearities. Specifically, we first prove the relative compactness, up to translation, of all minimizing sequences for the energy functional $E({\bf u})$ constrained on a suitable subset of $S(a_1,a_2)$. Secondly, we show global existence for solutions to \eqref{eqA0.1} with initial data close to $\mathcal G$.
\vskip1mm

\begin{remark}
\textup{(i)} Let us comment on the two ingredients described above in the context of equation \eqref{eqA0.2}. The first element   is proved along the analysis towards the existence of the ground states. In proving such existence result, due to the indefinite sign of the three wave interaction term in the corresponding energy functional, we  introduce an additional constrain given by an inequality, see \eqref{eq1.2}. Consequently, this makes appear further difficulties in proving the compactness of related minimizing sequences, which also marks a difference with respect to constrained variational problems with a sign-definite  structure, see for example  \cite{gl,GL,JJL,Soave2,WW}. Thus, we need to derive a better control of the mixed masses and suitable subset of $S(a_1,a_2)$ to guarantee that the minimizing sequence is far from the  boundary of the further constraint.\smallskip 

\noindent \textup{(ii)} As usual in the context of NLS equations with critical nonlinearity, it is not straightforward to establish a global theory. The method used in \cite{FLY} for \eqref{eqA0.2sub-crit} in the intercritical range $\frac{10}{3}\leq p<6$ does not work here. In particular, we cannot deduce global existence results from the a-priori estimates of $\sum^3_{j=1}\|\nabla u_j(t)\|^2_{L^2(\R^3)}$ that follow from the conservation laws. In the presence of the energy critical term, the local theory asserts that the time of existence for $H^1$-solutions depends instead on the profile of the initial data (see \cite[Theorem 4.5.1]{CT}). To overcome this difficulty, for energy critical nonlinear Schr\"{o}dinger equation, following \cite{CaW}, the authors of \cite{JJL} show that for initial data sufficiently close to ground state set the global existence holds. Hence, we first prove a uniform local existence result in our context, see Proposition \ref{pro1}. Using that the set $\mathcal{G}$ is compact, we show that for initial data sufficiently close to the set $\mathcal{G}$ global existence of solutions holds. With the help of these two elements, orbital stability of the set of ground states follows.
\end{remark}

For the proof of point  \textup{(ii)} of Theorem \ref{th1.2},  we further implement the strategy developed in our previous paper \cite{FLY}. Although the classical blow-up alternative seems unavailable in the energy critical setting, the virial-type estimate  established in \cite{FLY} still applies. Combining the conservation of mixed masses $Q_1$ and $Q_2$, we then obtain that the partial summation of gradient terms related to mixed masses blows-up in finite time. This leads to the strong instability. We remark that it remains open wether the global existence holds away from $\mathcal{G}$ or not.

\vskip1mm

We conclude by discussing the scattering of global solutions to \eqref{eqA0.1}.\smallskip

\noindent We recall that scattering for a global solution ${\bm \psi}(t)$ to \eqref{eqA0.1}   occurs if there  exist ${\bm \psi}_{+},{\bm \psi}_{-}\in H^1(\R^3,\mathbb{C}^3)$ such that
\begin{equation*}
\lim_{t\to -\infty}\|{\bm \psi}(t)- e^{it\Delta}{\bm \psi}_{-}\|_{H^1(\R^3, \mathbb C^3)}=0 \quad \text{and} \quad \lim_{t\to +\infty}\|{\bm \psi}(t)- e^{it\Delta}{\bm \psi}_{+}\|_{H^1(\R^3, \mathbb C^3)}=0.
\end{equation*}
Here $e^{it\Delta}$ stands for the linear Schr\"odinger propagator, and when applied to a vector function it is meant  to act component-wise.  \smallskip

In general, it is known that scattering does not always occur even for global solutions. Standing waves are basic example of global non-scattering solutions. As we will see in the next theorem, the presence of the three wave interaction term which is of mass-subcritical type prevents small data scattering for system \eqref{eqA0.1}. Indeed, as in Theorem \ref{th1.1} we show that ground states fulfills
$\sum^3_{j=1}\|\nabla u_j\|^2_{L^2(\R^3)}\to 0$ when $(a_1,a_2)\to (0^+,0^+)$, small data cannot scatter.
\begin{theorem}\label{th1.3}
Let $\alpha, a_1, a_2>0$ and $\max\{a_1,a_2\}<D$, then small data scattering cannot hold.
\end{theorem}

Theorem \ref{th1.3} shows a remarkable difference with respect to the classical energy critical NLS equation. Indeed, when $\alpha=0$, system \eqref{eqA0.1} reduces to the energy critical equation
\begin{equation}\label{energy-cri}
i\partial_tu+\Delta u+|u|^4u=0.
\end{equation}
Starting from the pioneering work \cite{CaW}, T. Cazenave and F. Weissler proved the local well-posedness and global well-posedness for small initial data. Hence, the presence of the three wave interaction term depicts a completely different scenario compared to \eqref{energy-cri}. It is worth mentioning that scattering for ``large data'' for solutions to \eqref{energy-cri} has been solved under certain regimes only recently by the famous concentration-compactness and rigidity argument by C.E. Kenig and F. Merle, see \cite{KM}. They proved global well-posedness and scattering for radial solutions with energy and kinetic energy less than those of ground state in dimensions $3\leq N\leq5$. Their result is sharp because the ground state (Aubin-Talenti function) does not scatter. R. Killip and M. Vi\c san \cite{KV}  extended the result of \cite{KM} to the non-radial case in dimension  $N\ge 5$.

\vskip1mm

\begin{remark}\rm
In light of the previous remark,  Theorem \ref{th1.3} indicates that the presence of the three wave interaction term prevents the occurrence of small data scattering. From the physical point of view, Theorems \ref{th1.2} also shows that the introduction of a linear coupling term leads a stabilization of a system which was originally unstable.
\end{remark}

\subsection{Notations}
In the paper, we use the following notations.  As the space dimension is fixed, working in the physical space $\R^3$, we simply write $L^p:=L^p(\R^3)$ with norm $\|f\|_{L^p}=\|f\|_p$ for  Lebesgue spaces,  $W^{1,p}:=W^{1,p}(\R^3)$ and $H^1:=H^1(\R^3)$ when $p=2$ for the classical Sobolev space, where  $f:\R^3\to\mathbb C$ or $f:\R^3\to\mathbb R$. For vector functions ${\bf f}:\R^3\to \mathbb C^3$ or ${\bf f}:\R^3\to \mathbb R^3$  where ${\bf f}=(f_1,f_2,f_3)$ we define  ${\bf L}^p:=L^p(\R^3, \mathbb C^3)$ or $L^p(\R^3, \mathbb R^3)$, and  ${\bf H}^1:=H^1(\R^3, \mathbb C^3)$ or $H^1(\R^3, \mathbb R^3)$, endowed with the following norms:
\[
\|{\bf f}\|_{p}^p:=\sum_{j=1}^3\|f_j\|_p^p
\]
and
\[
\|{\bf f}\|_{{\bf H}^1}^2:=\|{\bf f}\|_2^2+\|\nabla{\bf f}\|_2^2
\]
where $\nabla{\bf f}:=(\nabla f_1,\nabla f_2,\nabla f_3)$. We will often use the homogeneous space $\dot H^1$ endowed with the norm $\|f\|_{\dot H^1}=\|\nabla f\|_{2}$, and analogous notations in case of vector functions, i.e., $\|{\bf f}\|_{\dot{\bf H}^1}=\|\nabla{\bf f}\|_{2}$. Integrals  
 $\int_{\R^3} f  dx$ are simply denoted  by $\int f$. $\mathrm{Re}z$ and $\mathrm{Im}z$ are for the real and imaginary part of a complex number $z$, and $\overline z$ stands for the complex conjugate of $z$.

\section{Preliminary tools}
In this section, we give some preliminaries useful for the rest of the paper.
For $2<p<6$, let us recall  the Gagliardo-Nirenberg inequality 
\begin{equation*} 
 \|u\|_{p} \leq C_{\GN}(p) \|\nabla u\|_{2}^{\gamma_p} \|u\|_{2}^{1-\gamma_p}, \quad \forall\,  u \in {H}^{1},
\end{equation*}
where $C_{\GN}(p)$ is the best constant in the Gagliardo-Nirenberg-Sobolev inequality $H^{1} \hookrightarrow L^{p}$  in $\R^3$, and $\gamma_p=\frac{3(p-2)}{2p}$. For $p=3$, we have $C_{\GN}(3)=\left(\frac{2}{\|W\|_2}\right)^{\frac{1}{3}}$,  where $W$ is defined in \eqref{eq:w}.\\

We start by recalling the following Lemma, giving a Pohozaev identity.

\begin{lemma} \label{lem2.1}
Let ${\bf u}\in {\bf H}^1$ be a solution to \eqref{eqA0.2}. Then the following identity holds true:
\begin{equation*} 
P({\bf u}):=\|\nabla{\bf u}\|_2^2-\|{\bf u}\|_6^6-\frac32\alpha \mathrm{Re} \int u_1u_2\overline{u}_3=0.
\end{equation*}
\end{lemma}
\begin{proof}
The proof is classical, and we refer the reader to \cite{BL}. 
\end{proof}

We now introduce the $L^2$-norm-preserving dilation operator
\begin{equation*}
s\star {\bf u}(x):=\left(s^{\frac{3}{2}}u_1(sx),s^{\frac{3}{2}}u_2(sx),s^{\frac{3}{2}}u_3(sx)\right)
\end{equation*}
with $s>0$. As $\lim\limits_{s\to \infty}E(s\star {\bf u})=-\infty$, we see that $\inf\limits_{{\bf u}\in S(a_1,a_2)}E({\bf u})=-\infty$. Furthermore, we introduce the Pohozaev set
\begin{equation}
\mathcal{P}_{a_1,a_2}:=\left\{{\bf u}\in S(a_1,a_2): P({\bf u}):=\|\nabla{\bf u}\|_2^2-\|{\bf u}\|_6^6-\frac32\alpha \mathrm{Re} \int u_1u_2\overline{u}_3=0\right\}.
\end{equation}

The Pohozaev set $\mathcal{P}_{a_1,a_2}$ is related to the fiber maps
\begin{equation}\label{eq:fiber}
\Psi_{{\bf u}}(s)=E(s\star{\bf u})=\frac{s^2}{2}\|\nabla{\bf u}\|_2^2-\frac{s^{6}}{6}\|{\bf u}\|_6^6-s^{\frac{3}{2}}\alpha\mathrm{Re}\int u_1u_2\overline{u}_3.
\end{equation}
Indeed, we have $s\Psi'_{{\bf u}}(s)=P(s\star{\bf u})$.
Note that $\mathcal{P}_{a_1,a_2}$ can be divided into the disjoint union $\mathcal{ P}_{a_1,a_2}=\mathcal{ P}_{a_1,a_2}^+\cup \mathcal{ P}_{a_1,a_2}^0\cup \mathcal{ P}_{a_1,a_2}^-$, where
\begin{equation}\label{eq:c41}
\begin{aligned}
	\mathcal{ P}_{a_1,a_2}^+&:=\left\{{\bf u}\in \mathcal{ P}_{a_1,a_2}  \ \hbox{ s.t. }\ \Psi_{{\bf u}}''(1)>0\right\},\\
	\mathcal{ P}_{a_1,a_2}^0&:=\left\{{\bf u}\in \mathcal{ P}_{a_1,a_2}  \ \hbox{ s.t. } \ \Psi_{{\bf u}}''(1)=0\right\},\\
	\mathcal{ P}_{a_1,a_2}^-&:=\left\{{\bf u}\in \mathcal{ P}_{a_1,a_2} \  \hbox{ s.t. } \ \Psi_{{\bf u}}''(1)<0\right\}.
\end{aligned}
\end{equation}

To show that the energy functional $E|_{S(a_1,a_2)}$ has a concave-convex geometry (i.e., a structure with a local minimum and a global maximum, where the local minimum is strictly less than zero and the global maximum is strictly greater than zero; see Lemma \ref{lem2.11} below), we introduce the following constraint
\begin{equation}\label{eq1.2}
\mathcal{M}:=\left\{{\bf u}\in {\bf H}^1 \ \hbox{ s.t. } \ \mathrm{Re}\int u_1u_2\overline{u}_3 >0\right\},
\end{equation}
and then define 
\begin{equation*}
m^{+}(a_1,a_2)=\inf_{{\bf u}\in \mathcal{P}^{+}_{a_1,a_2}\cap \mathcal{M}} E({\bf u})
\end{equation*}
and 
\begin{equation*}
m^{-}(a_1,a_2)=\inf_{{\bf u}\in \mathcal{P}^{-}_{a_1,a_2}\cap \mathcal{M}} E({\bf u}).
\end{equation*}

In the spirit of Wei and Wu \cite{WW}, for ${\bf u}\in \mathcal{M}$, we see that the presence of the mass subcritical term $\mathrm{Re}\int u_1u_2\overline{u}_3$ induces a convex-concave geometry of $E|_{S(a_1,a_2)}$ if $\alpha>0$ and $a_1,a_2>0$ are sufficiently small.\\

For ${\bf u}\in S(a_1,a_2)$, it is immediate to see that  $\|u_1\|_2\leq a_1$, $\|u_2\|_2\leq a_2$ and $\|u_3\|_2\leq\min\{a_1,a_2\}$.
By Sobolev's inequality and Young's inequality, we have
\begin{equation}\label{eq:z5}
\begin{aligned}
\frac{1}{6}\|{\bf u}\|^6_6\leq\frac{1}{6}C_{\Sob}^{-3}\sum^{3}_{j=1}\|\nabla u_j\|^{6}_2\leq A_1\|{\bf u}\|_{{\bf {\dot H}}^1}^{6},
\end{aligned}
\end{equation}
where $A_1:=\frac{1}{6}C_{\Sob}^{-3}$.
Similarly, we have
\begin{equation}\label{eq:z4}
\begin{aligned}
\left|\alpha\mathrm{Re}\int u_1u_2\overline{u}_3\right|&\leq\alpha \int |u_1||u_2||u_3|
\\
&\leq\frac{2\alpha}{3^{\frac{3}{4}}\|W\|_2} \max\{a_1,a_2\}^{\frac{3}{2}}\|{\bf u}\|_{{\bf {\dot H}}^1}^{\frac{3}{2}}=A_2\|{\bf u}\|_{{\bf {\dot H}}^1}^{\frac{3}{2}},
\end{aligned}
\end{equation}
where $A_2:=\frac{2\alpha}{3^{\frac{3}{4}}\|W\|_2}\max\{a_1,a_2\}^{\frac{3}{2}}$.
Then, combining \eqref{eq:z5} and \eqref{eq:z4} with the definition of the energy,  we get
\begin{equation*}\label{b3}
\begin{aligned}
E({\bf u})&=\frac{1}{2}\|{\bf u}\|_{{\bf {\dot H}}^1}^2-\frac{1}{6}\|{\bf u}\|_{{\bf L}^6}^6-\alpha \mathrm{Re}\int u_1u_2\overline{u}_3\\
&\geq \frac{1}{2}\|{\bf u}\|_{{\bf {\dot H}}^1}^2-A_1\|{\bf u}\|_{{\bf {\dot H}}^1}^6-A_2\|{\bf u}\|_{{\bf {\dot H}}^1}^{\frac{3}{2}}=:h(\|{\bf u}\|_{{\bf {\dot H}}^1}),
\end{aligned}
\end{equation*}
where
\begin{equation}\label{def:func-h}
h(\rho)=\frac{\rho^2}{2}-A_1\rho^{6}-A_2\rho^{\frac{3}{2}}.
\end{equation}
\vskip2mm

The next Lemma  below  shows that the functional $E$  has a concave-convex structure on $S(a_1,a_2)$.

\begin{lemma}\label{lem2.11}
Let $\alpha, a_1, a_2>0$. Let $D$ be as in \eqref{def:D} and $h$ as in \eqref{def:func-h}.\smallskip

\noindent \textup{(i)} If $\max\{a_1,a_2\}< D$, then $h(\rho)$ has a local minimum at negative level and a global maximum at positive level. Moreover, there exist $R_0=R_0(a_1,a_2)$, $R_1=R_1(a_1,a_2)$, and $\rho_0$ such that, $R_0<\max\{a_1,a_2\}D^{-1}\rho_0<\rho_0<R_1$, and
\begin{equation*}
h(R_0)=h(R_1)=0, \quad h(\rho)>0 \iff \rho \in (R_0,R_1).
\end{equation*}

\noindent \textup{(ii)} If $\max\{a_1,a_2\}=D$,  then $h(\rho)$ has a local minimum at negative level and a global maximum at level zero.
Moreover, we have
\[
h(\rho_0)=0 \quad \hbox{ and } \quad h(\rho)<0 \iff \rho\in(0, \rho_0) \cup (\rho_0,+\infty).
\]
\end{lemma}

\begin{proof}
\textup{(i)} We first prove that $h$ has exactly two critical points. Indeed,
\[
h'(\rho)=0 \Longleftrightarrow \hat h(\rho)=\frac{3A_2}{2}, \quad \mbox{with} \quad \hat h(\rho)=\rho^{\frac{1}{2}}-6A_1\rho^{\frac{9}{2}}.
\]
We have that $\hat h(\rho)$ is increasing on $[0,\bar{\rho})$ and decreasing on $(\bar{\rho},+\infty)$, with the point $\rho$ being $\bar{\rho}=\left(\frac{1}{54A_1}\right)^{\frac{1}{4}}$.
We get
\[
\max_{\rho\geq0}\hat h(\rho)=\hat h(\bar{\rho})
=\frac{8}{9}\left(\frac{1}{54A_1}\right)^{\frac{1}{8}}
>\frac{3A_2}{2}
\]
if and only if
\begin{equation*}
\begin{aligned}
\max\{a_1,a_2\}<D_0:=\frac{4C_{\Sob}^{\frac{1}{4}}\|W\|^{\frac{2}{3}}_2}{3^{\frac{5}{3}} \alpha^{\frac{2}{3}}}.
\end{aligned}
\end{equation*}
As $\lim\limits_{s\to 0^+}\hat h(s)=0^+$ and $\lim\limits_{s\to +\infty}\hat h(s)=-\infty$, we see  that $h$ has exactly two critical points if $\max\{a_1,a_2\}<D_0$.
\vskip1mm
\noindent Note that
\[
h(\rho)>0 \Longleftrightarrow \tilde h(\rho)>A_2\quad \hbox{ with }\quad  \tilde h(\rho)=\frac{1}{2}\rho^{\frac{1}{2}}-A_1\rho^{\frac{9}{2}}.
\]
It is not difficult to check that $ \tilde h(\rho)$ is increasing on $[0,\rho_0)$ and decreasing on $(\rho_0,+\infty)$, where
\begin{equation*}\label{radial}
\rho_0=\left(\frac{1}{18A_1}\right)^{\frac{1}{4}}.
\end{equation*}
We have
\[
\max_{\rho\geq0} \tilde h(\rho)= \tilde h(\rho_0)
=\frac{4}{9}\left(\frac{1}{18A_1}\right)^{\frac{1}{8}}
>A_2
\] provided
\begin{equation*}
\begin{aligned}
\max\{a_1,a_2\}<D: =\frac{2^{\frac{2}{3}}C_{\Sob}^{\frac{1}{4}}\|W\|^{\frac{2}{3}}_2}{3^{\frac{11}{12}}\alpha^{\frac{2}{3}}}.
\end{aligned}
\end{equation*}
We also  have that $h(\rho)>0$ on an open interval $(R_0,R_1)$ if and only if $\max\{a_1,a_2\}<D$.
By direct calculations, we get that $D<D_0$.
\vskip1mm

\begin{equation*}
\begin{aligned}
\tilde h\left(\frac{\max\{a_1,a_2\}}{D}\rho_0\right)
&=\frac{1}{2}\left(\frac{\max\{a_1,a_2\}}{D}\right)^{\frac{1}{2}}\rho^{\frac{1}{2}}_0
-A_1\left(\frac{\max\{a_1,a_2\}}{D}\right)^{\frac{9}{2}}\rho^{\frac{9}{2}}_0\\
&>\frac{4}{9}\left(\frac{\max\{a_1,a_2\}}{D}\right)^{\frac{1}{2}}\left(\frac{1}{18A_1}\right)^{\frac{1}{8}}
>A_2.
\end{aligned}
\end{equation*}
\vskip1mm

\noindent \textup{(ii)} Similarly to the proof of \textup{(i)}, we have
\begin{equation*}
R_0=\rho_{0}=R_1,\quad \tilde h(\rho_{0})=A_2,\quad \hat h(\bar{\rho})>\frac{3A_2}{2}.
\end{equation*}
\end{proof}

In what follows, we study the structure of the manifold
\begin{equation}\label{def:bar-p}
\bar{\mathcal{P}}_{a_1,a_2}:=\mathcal{P}_{a_1,a_2}\cap \mathcal{M}.
\end{equation}
We will observe that a critical point for the energy  functional $E$ on $\bar{\mathcal{P}}_{a_1,a_2}$ is a critical point for the same functional  on $S(a_1,a_2)$. Therefore, $\bar{\mathcal{P}}_{a_1,a_2}$ is a natural constraint.
\begin{lemma}\label{lem2.2}
Let $\alpha, a_1, a_2>0$. If $\max\{a_1,a_2\}\leq D$, then $\mathcal{P}^0_{a_1,a_2}=\emptyset$, and the set $\bar{\mathcal{P}}_{a_1,a_2}$ is a $ C^1$-submanifold of codimension 1 in $S(a_1,a_2)$.
\end{lemma}

\begin{proof}
It is sufficient to prove that $\mathcal{P}^0_{a_1,a_2}$ is empty. Indeed, if $\mathcal{P}^0_{a_1,a_2}=\emptyset$, we show that $\bar{\mathcal{P}}_{a_1,a_2}$ is a $C^1$-submanifold of codimension 1 in $S(a_1,a_2)$.
Assume by contradiction that there exists a ${\bf u} \in \mathcal{P}^0_{a_1,a_2}$ such that $P({\bf u})=0$, thus
\begin{equation*}
\Psi''_{{\bf u}}(1)=\|{\bf u}\|_{{\bf {\dot H}}^1}^2-6\|{\bf u}\|_{6}^6
-\frac{9}{4}\alpha\mathrm{Re}\int u_1u_2\overline{u}_3=0.
\end{equation*}
Let
\begin{equation*}
\begin{aligned}
f(s):&=s\Psi'_{{\bf u}}(1)-\Psi''_{{\bf u}}(1)\\
&=(s-2)\|{\bf u}\|_{{\bf {\dot H}}^1}^2-(s-6)\|{\bf u}\|_{6}^6-\frac{3}{2}\left(s-\frac{3}{2}\right)\alpha\mathrm{Re}\int u_1u_2\overline{u}_3,\\
\end{aligned}
\end{equation*}
and observe that $f(s)=0$, $\forall s\in \R$. Therefore, it follows from $f\left(\frac{3}{2}\right)=0$ that
\begin{equation}\label{eq:z6}
\|{\bf u}\|_{{\bf {\dot H}}^1}^2=9\|{\bf u}\|_{6}^6.
\end{equation}
By \eqref{eq:z5} and $\eqref{eq:z6}$, we have
\begin{equation*}
\|{\bf u}\|_{{\bf {\dot H}}^1}^4\ge \frac{C_{\Sob}^3}{9}.
\end{equation*}
Since $f(6)=0$, we get
\begin{equation*}
\begin{aligned}
4=\frac{27\alpha}{4\|{\bf u}\|_{{\bf {\dot H}}^1}^2}\mathrm{Re}\int u_1u_2\overline{u}_3\leq \frac{3^{\frac{5}{2}}\alpha }{2C_{\Sob}^{\frac{3}{8}}\|W\|_2 }\max\{a_1,a_2\}^{\frac{3}{2}},
\end{aligned}
\end{equation*}
which is a contradiction with respect  to the hypothesis $\max\{a_1,a_2\}\leq D<D_0$.\vskip2mm
\noindent  We omit the proof that $\bar{\mathcal{P}}_{a_1,a_2}$ is a smooth manifold of codimension 1 in $S(a_1,a_2)$.
\end{proof}

\begin{lemma}\label{lem2.3}
Let $\alpha, a_1, a_2>0$. If $\max\{a_1,a_2\}< D$, for ${\bf u}\in S(a_1,a_2)\cap \mathcal{M}$, then the function $\Psi_{{\bf u}}(s)$ has exactly two critical points $s_{{\bf u}}<\sigma_{{\bf u}}\in\R$ and two zeros $c_{{\bf u}}<d_{{\bf u}}$ with $s_{{\bf u}}<c_{{\bf u}}<\sigma_{{\bf u}}<d_{{\bf u}}$. Moreover, we have the properties below:\vskip2mm

\noindent \textup{(i)}  $s_{{\bf u}}\star{\bf u}\in\mathcal{P}^{+}_{a_1,a_2}$ and $\sigma_{{\bf u}}\star{\bf u}\in\mathcal{P}^{-}_{a_1,a_2}$. Moreover,  if $s\star{\bf u}\in \mathcal{P}_{a_1,a_2}$, then either $s=s_{{\bf u}}$   or $s=\sigma_{{\bf u}}$,\vskip2mm

\noindent \textup{(ii)}  $s_{{\bf u}}<R_0\|{\bf u}\|_{{\bf {\dot H}}^1}^{-1}$ and
\begin{equation*}
\Psi_{{\bf u}}(s_{{\bf u}})=\inf\left\{\Psi_{{\bf u}}(s):s\in\left(0,R_0\|{\bf u}\|_{{\bf {\dot H}}^1}^{-1}\right)\right\}<0,
\end{equation*}\vskip2mm

\noindent \textup{(iii)}  $E\left(\sigma_{{\bf u}}\star{\bf u}\right)=\max\limits_{s\in\R}E\left(s\star{\bf u}\right)>0$,\vskip2mm

\noindent\textup{(iv)} The maps ${\bf u} \mapsto s_{{\bf u}} \in \mathbb{R}$ and ${\bf u} \mapsto \sigma_{{\bf u}} \in \mathbb{R}$ are of class $ C^1$.
\end{lemma}

\begin{proof}
The proof follows the same lines of  \cite[Lemma 2.4]{FLY} by the authors, with obvious modifications.
\end{proof}

\section{Proof of Theorem \ref{th1.1}}

In this section, we give a  proof of  Theorem \ref{th1.1}. We first prove several results eventually leading to the conclusions of the Theorem. \\

First of all, we define the ball in ${\bf H}^1$ 
\begin{equation*}
B_{\rho_0}:=\left\{{\bf u}\in {\bf H}^1 \ \hbox{ s.t. } \ \|{\bf u}\|_{{\bf {\dot H}}^1}<\rho_0\right \}
\end{equation*}
and its subset given by 
\begin{equation}\label{def:V}
V(a_1,a_2):=S(a_1,a_2)\cap B_{\rho_0}\cap \mathcal{M},
\end{equation}
where $\mathcal{M}$ is defined in \eqref{eq1.2}.\\

We aim at minimizing the energy functional $E$ over the set introduced in \eqref{def:V}, provided that the two positive parameters  $a_1$ and $a_2$ are such that $\max\{a_1,a_2\}< D$. We then define  
\begin{equation}\label{def:m12}
m(a_1,a_2) := \inf_{{\bf u} \in V(a_1,a_2)} E({\bf u}).
\end{equation}

First, we claim that the minimization problem \eqref{def:m12} is equivalent to the minimization of the energy functional over different manifolds, and that the infimum is strictly negative. Recall the definition of $\mathcal{P}^+_{a_1,a_2}$ and $\bar{\mathcal{P}}_{a_1,a_2}$ in \eqref{eq:c41} and \eqref{def:bar-p}, respectively.

\begin{lemma}\label{lem2.4}
Let $\alpha, a_1, a_2>0$. If $\max\{a_1,a_2\}< D$, the set $\mathcal{P}^+_{a_1,a_2}$ is contained in $V(a_1,a_2)$ and
\begin{equation}
m(a_1,a_2)=m^+(a_1,a_2):=\inf_{{\bf u}\in\mathcal{P}^+_{a_1,a_2}\cap \mathcal{M}}E({\bf u})=\inf_{{\bf u}\in\bar{\mathcal{P}}_{a_1,a_2}}E({\bf u})<0.
\end{equation}
Moreover, there exists $\varepsilon_0>0$ such that for any $0<\varepsilon<\varepsilon_0$
\begin{equation*}
m(a_1,a_2)<\inf_{S(a_1,a_2)\cap (B_{\rho_0}\setminus B_{\rho_0-\varepsilon})} E({\bf u}).
\end{equation*}
\end{lemma}
\begin{proof}
The Lemma can be proved along the
same lines of  \cite[Lemma 3.2]{FLY} by the authors, with obvious modifications.
\end{proof}
 
We now introduce some other notions and tools. Let ${\bf u}$ belong to ${\bf H}^1$, and let us use the short notation $|{\bf u}|$ standing for  $|{\bf u}|=(|u_1|,|u_2|,|u_3|)$. Firstly, we have  $E(|{\bf u}|)\leq E({\bf u})$. Moreover, by  symmetric rearrangement, see \cite{BF,LE}, we also claim that 
\[
\|\nabla |u_j|^*\|_2\leq\|\nabla |u_j|\|_2\leq\|\nabla u_j\|_2, \quad \||u_j|^*\|_p=\|u_j\|_p, \]
and
\[
\int |u_1||u_2||u_3|\leq\int |u_1|^*|u_2|^*|u_3|^*,
\]
where $|u_j|^*$ is the Schwarz  symmetric rearrangement of $|u_j|$, for $j\in\{1,2,3\}$, and $|{\bf u}|^*=(|u_1|^*,|u_2|^*,|u_3|^*)$. Then $E(|{\bf u}|^*)\leq E(|{\bf u}|)\leq E({\bf u})$.  
Let us now consider  ${\bf v}\in {\bf H}^1$ a solution to the  system \eqref{eqA0.2} with $\lambda_3=\lambda_1+\lambda_2$.  Precisely, ${\bf v}=(v_1,v_2,v_3)$ solves the system 
\begin{equation}\label{eq:b4}
\begin{cases}
-\Delta v_{1}+ \lambda_1v_{1}=\left|v_{1}\right|^{4} v_{1}+\alpha v_{3} v_{2},\\
-\Delta v_{2}+ \lambda_2v_{2}=\left|v_{2}\right|^{4} v_{2}+\alpha v_{3} v_{1}, \\
-\Delta v_{3}+ (\lambda_1+\lambda_2)v_{3}=\left|v_{3}\right|^{4} v_{3}+\alpha v_{1} v_{2}.
\end{cases}
\end{equation}
With ${\bf H}^1_{\rad}$ standing for the subspace of functions in ${\bf H}^1$ which are radially symmetric component-wise, we introduce the manifolds
\[
\mathcal{P}_{\rad,a_1,a_2}:=\left\{{\bf v}\in {\bf H}^1_{\rad}\cap S(a_1,a_2) \ \hbox{ s.t. } \ \mathcal P({\bf v})=0\right\},
\]
and
\begin{equation*}\label{def:P-rad-pm}
\begin{aligned}
\mathcal{P}^{+}_{\rad,a_1,a_2}:={\bf H}^1_{\rad}\cap \mathcal{P}^{+}_{a_1,a_2},\\ \mathcal{P}^{-}_{\rad,a_1,a_2}:={\bf H}^1_{\rad}\cap \mathcal{P}^{-}_{a_1,a_2}.
\end{aligned}
\end{equation*}
Subsequently, we introduce the minimization problems 
\begin{equation*}\label{y2}
\begin{aligned}
m^{+}_{r}(a_1,a_2):=\inf\limits_{{\bf u}\in \mathcal{P}^{+}_{\rad,a_1,a_2}\cap \mathcal{M}}E({\bf u}), \\ 
m^{-}_{r}(a_1,a_2):=\inf\limits_{{\bf u}\in \mathcal{P}^{-}_{\rad,a_1,a_2}\cap \mathcal{M}}E({\bf u}).
\end{aligned}
\end{equation*}

With these tools at hand, we can state the following.

\begin{lemma}\label{lem2.5}
Let $\alpha,a_1, a_2>0$. If $\max\{a_1,a_2\}<D$, then
\begin{equation*}
m^{+}_{r}(a_1,a_2)=\inf_{{\bf u}\in \mathcal{P}^{+}_{\rad,a_1,a_2}\cap \mathcal{M}} E({\bf u})=\inf_{{\bf u}\in\mathcal{P}^{+}_{a_1,a_2}\cap \mathcal{M}} E({\bf u}),
\end{equation*}
and
\begin{equation*}
m^{-}_{r}(a_1,a_2)=\inf_{{\bf u}\in \mathcal{P}^{-}_{\rad,a_1,a_2}\cap \mathcal{M}} E({\bf u})=\inf_{{\bf u}\in\mathcal{P}^{-}_{a_1,a_2}\cap \mathcal{M}} E({\bf u}).
\end{equation*} 
Furthermore, $\inf\limits_{\mathcal{P}^{+}_{a_1,a_2}\cap \mathcal{M}} E$ is reached by a vector function $(e^{i\theta_1}w_1, e^{i\theta_1}w_2, e^{i(\theta_1+\theta_2)}w_3 )$ where ${\bf w}=(w_1,w_2,w_3)$ is a minimizer for $\inf\limits_{\mathcal{P}^{+}_{\rad,a_1,a_2}} E$, and $(\theta_1,\theta_2)$ are two real parameters. Similarly,  $\inf\limits_{\mathcal{P}^{-}_{a_1,a_2}\cap \mathcal{M}} E$ is reached by a vector function $(e^{i\theta_1}w_1, e^{i\theta_1}w_2, e^{i(\theta_1+\theta_2)}w_3 )$ where $\tilde{\bf w}=(\tilde w_1, \tilde w_2,\tilde w_3)$ is a minimizer for $\inf\limits_{\mathcal{P}^{-}_{\rad,a_1,a_2}} E$,  and $(\tilde \theta_1,\tilde \theta_2)$ are two real parameters.
\end{lemma}
\begin{proof}
    The proof is analogous to that of  \cite[Lemma 3.3]{FLY} by the authors, and we omit it.
\end{proof}

We now give the existence of a ground state solution to \eqref{eqA0.2} along with some of its properties. 

\begin{lemma}\label{lem2.6}
Let $\alpha, a_1, a_2>0$. If $\max\{a_1,a_2\}<D$, then \eqref{eqA0.2} has a ground state solution $(\lambda_1,\lambda_2,u_1,u_2,u_3)$ with $\lambda_1,\lambda_2>0$, and ${\bf u}\in S(a_1,a_2)$ is real valued, positive and radially symmetric. 
\end{lemma}
\begin{proof}
By Lemma \ref{lem2.5}, it suffices to demonstrate that $m^+_r(a_1,a_2)$ is achieved.
Given that $\displaystyle m^+_r(a_1,a_2)=\inf_{V(a_1,a_2)}E$,  and employing the symmetric decreasing rearrangement, we obtain a minimizing sequence $\{{\bf w}_n\}$ with ${\bf w}_n\in {\bf H}^1_{\rad}\cap V(a_1,a_2)$ which is positive for every $n$. Furthermore, by Lemma \ref{lem2.4}, $E(s_{{\bf w}_n}\star {\bf w}_n)\leq E({\bf w}_n)$ and $s_{{\bf w}_n}\star {\bf w}_n\in V(a_1,a_2)$.
By replacing ${\bf w}_n$ by $s_{{\bf w}_n}\star {\bf w}_n$, we obtain a new minimizing sequence $s_{{\bf w}_n}\star {\bf w}_n\in \mathcal{P}^+_{a_1,a_2}\cap \mathcal{M}$. Hence, by  Ekeland's variational principle,
we can select a non-negative radial Palais-Smale sequence $\{{\bf u}_n\}$ for $E|_{S(a_1,a_2)}$ at the level $m^+_r(a_1,a_2)$ with $P({\bf u}_n)=o_n(1)$ such that $\lim\limits_{n\to \infty} E({\bf u}_n)=m^+_r(a_1,a_2)$ and $E'|_{S(a_1,a_2)}({\bf u}_n)\to 0$ as $n\to \infty$. Since
\begin{equation*}
\begin{aligned}
m^+_r(a_1,a_2)+o_n(1)=E({\bf u}_n)&=\frac{1}{3}\|{\bf u}_n\|_{{\bf {\dot H}}^1}^2-\frac{3}{4}\alpha \int u_{1,n}u_{2,n}u_{3,n},\\
\end{aligned}
\end{equation*}
the sequence $\{{\bf u}_n\}$ is bounded in $H^1_{\rad}(\R^3,\R^3)$. Indeed, since $m^+_r(a_1,a_2)<0$, by the H\"older  and the  Gagliardo-Nirenberg inequalities,
\begin{equation}\label{eq:b7-prev}
\|{\bf u}_n\|_{{\bf {\dot H}}^1}^2\leq\frac{9}{4}\alpha \int u_{1,n}u_{2,n}u_{3,n} \leq\frac{  3^{\frac{5}{4}} \alpha}{2\|W\|_2} \max\{a_1,a_2\}^{\frac{3}{2}}\|{\bf u}_n\|_{{\bf {\dot H}}^1}^{\frac{3}{2}}.
\end{equation}
Hence $\{{\bf u}_n\}$ is bounded in ${\bf H}^1_{\rad}$. Then there exists ${\bf u}=(u_1,u_2,u_3)$ such that ${\bf u}_n\rightharpoonup{\bf u}$ weakly in ${\bf H}^1_{\rad}$, ${\bf u}_n\to{\bf u}$ strongly in ${\bf L}^{r}$ for $r\in (2, 6)$, and a.e. in $\R^3\times \R^3\times \R^3$ as $n\to \infty$. Therefore, $u_j$ are non-negative radial functions for  $j\in\{1,2,3\}$.
\vskip1mm

According to the Lagrange multiplier's rule (refer to \cite[Lemma 3]{BL}), there exists a sequence $\{(\lambda_{1,n},\lambda_{2,n})\}\subset \R\times\R$ such that
\begin{equation}\label{eq:b7}
\begin{cases}
\displaystyle\int\left( \nabla u_{1,n}\nabla\phi_1+\lambda_{1,n}u_{1,n}\phi_1-|u_{1,n}|^{4}u_{1,n}\phi_1-\alpha u_{3,n}u_{2,n}\phi_1\right)=o_n(1)\|\phi_1\|_{H^1},\\
\displaystyle\int \left(\nabla u_{2,n}\nabla\phi_2+\lambda_{2,n}u_{2,n}\phi_2-|u_{2,n}|^{4}u_{2,n}\phi_2-\alpha u_{3,n}u_{1,n}\phi_2\right)=o_n(1)\|\phi_2\|_{H^1},\\
\displaystyle\int \left(\nabla u_{3,n}\nabla\phi_3+(\lambda_{1,n}+\lambda_{2,n})u_{3,n}\phi_3-|u_{3,n}|^{4}u_{3,n}\phi_3-\alpha u_{1,n}u_{2,n}\phi_2\right)=o_n(1)\|\phi_3\|_{H^1},
\end{cases}
\end{equation}
as $n\to \infty$, for every $\phi_j\in H^1,$ $j\in\{1,2,3\}$. In particular, by taking $(\phi_1,\phi_2,\phi_3)=(u_{1,n},u_{2,n},u_{3,n})$, we have  that $(\lambda_{1,n},\lambda_{2,n})$ is bounded, and up to a subsequence $(\lambda_{1,n},\lambda_{2,n})\to (\lambda_{1},\lambda_{2})\in \R^2$. Passing to the limit in \eqref{eq:b7}, we get 
\begin{equation*}
\begin{cases}
-\Delta u_1+\lambda_1u_1=|u_1|^{4}u_1+\alpha u_3u_2,\\
-\Delta u_2+\lambda_2u_2=|u_2|^{4}u_2+\alpha u_3u_1,\\
-\Delta u_3+(\lambda_1+\lambda_2)u_3=|u_3|^{4}u_3+\alpha u_1u_2.
\end{cases}
\end{equation*}
Furthermore, we infer that $\mathrm{Re}\int u_{1}u_2\overline{u}_3>0$. Supposing that this is not true, and by using the Sobolev embedding, we get   
\begin{equation*}
\|{\bf u}\|_{{\bf {\dot H}}^1}^2\leq \|{\bf u}\|_{6}^6 \leq C_{\Sob}^{-3}\|{\bf u}\|_{{\bf {\dot H}}^1}^6,
\end{equation*}
and then $C_{\Sob}^{\frac{3}{2}}\leq\|{\bf u}\|_{{\bf {\dot H}}^1}^2$. Moreover, as $\mathcal{P}^+_{a_1,a_2}\subset V(a_1,a_2)$, we get ${\bf u}\in B_{\rho_0}$, and this is a contradiction. \\
From $P({\bf u})=0$, we conclude that
\begin{equation}\label{eq:d3}
\lambda_1\|u_1\|^2_2+\lambda_2\|u_2\|^2_2+(\lambda_1+\lambda_2)\|u_3\|^2_2=\frac{3\alpha}{2}\int u_1u_2u_3.
\end{equation}
By $P({\bf u}_n)=o_n(1)$, we obtain
\begin{equation}\label{eq:d4}
\begin{aligned}
\lambda_1a^2_1+\lambda_2a^2_2&=\lim_{n\to\infty}\left(\lambda_1\|u_{1,n}\|^2_2+\lambda_2\|u_{2,n}\|^2_2+(\lambda_1+\lambda_2)\|u_{3,n}\|^2_2\right)\\
&=\lim_{n\to\infty}\frac{3\alpha}{2}\int u_{1,n}u_{2,n}u_{3,n}=\frac{3\alpha}{2}\int u_{1}u_{2}u_{3}.
\end{aligned}
\end{equation}
\vskip1mm

\noindent We now claim that $u_j\not\equiv0$ for any $j\in\{1,2,3\}$.
\vskip2mm

\noindent Indeed, if there exists a $j \in \{1,2,3\}$ such that $u_j=0$, then $\int u_{1,n}u_{2,n}u_{3,n}\to 0$,and by  definition of $P({\bf u}_n)$ we have
\begin{equation*}
\|{\bf u}_n\|_{{\bf {\dot H}}^1}^2-\|{\bf u}_n\|_{6}^6=o_n(1).
\end{equation*}
It follows that
$\lim\limits_{n\to\infty}\|{\bf u}_n
\|_{{\bf {\dot H}}^1}^2
\geq C_{\Sob}^{\frac{3}{2}}$. We then have
\begin{equation*}
\begin{aligned}
m^+(a_1,a_2)&=E({\bf u}_n)-\frac{1}{6}P({\bf u}_n)+o_n(1)\\
&=\frac{1}{3}\|{\bf u}_n\|_{{\bf {\dot H}}^1}^2
-\frac{3\alpha}{4}\int u_{1,n}u_{2,n}u_{3,n}+o_n(1)\\
&\geq \frac13\|{\bf u}_n\|_{{\bf {\dot H}}^1}^2 \ge \frac{1}{3}C_{\Sob}^{\frac{3}{2}},
\end{aligned}
\end{equation*}
and this is a contradiction with respect to  $m^+(a_1,a_2)<0$.
\vskip1mm

It is left to prove that $m^+_r(a_1,a_2)$ is achieved.
From \cite[Lemma A.2]{IN}, we get $\lambda_1,\lambda_2>0$. Moreover, combining \eqref{eq:d3} with \eqref{eq:d4}, the following identity holds:
\begin{equation}\label{eq:d51}
\lambda_1a^2_1+ \lambda_2a^2_2=\lambda_1\|u_1\|^2_2+\lambda_2\|u_2\|^2_2+(\lambda_1+\lambda_2)\|u_3\|^2_2.
\end{equation}
Since $\|u_1\|^2_2+\|u_3\|^2_2\leq a^2_1$ and $\|u_2\|^2_2+\|u_3\|^2_2\leq a^2_2$, it follows from \eqref{eq:d51} that $\|u_1\|^2_2+\|u_3\|^2_2=a^2_1$ and $\|u_2\|^2_2+\|u_3\|^2_2=a^2_2$, and hence ${\bf u}\in \mathcal{P}_{\rad,a_1,a_2}$. By the maximum principle (see \cite[Theorem 2.10]{hanq}), $u_j>0$, $j\in\{1,2,3\}$. We then conclude that ${\bf u}_n\to {\bf u}$ in ${\bf H}^1_{\rad}$ and $E({\bf u})=m^+_r(a_1,a_2)$.
\end{proof}

In the subsequent discussion, we establish a refined upper bound for $m^+_r(a_1,a_2)$ under the condition that $a_1 = a_2$. Specifically, Lemma \ref{lem3.5} below demonstrates that $m^+_r(a_1,a_1)$ is not only negative but also maintains a distance from zero. We introduce the problem
\begin{equation}\label{eq:g1}
\begin{cases}
-\Delta u +\lambda u =\alpha u^2,\\
\displaystyle\int |u|^2=a^2,
\end{cases}
\end{equation}
where $\alpha,a>0$ are fixed. Denote by $J_0(u)$ the following energy:
\begin{equation*}\label{def:J0}
J_0(u)=\frac{1}{2}\| \nabla u\|^2_{2}-\frac{\alpha}{3}\|u\|^{3}_{3}.
\end{equation*}
A solution $u$ to \eqref{eq:g1} can be found as a minimizer of
\begin{equation}\label{eq:e1}
0>m_0(a):=\inf_{u\in S(a)} J_0(u)>-\infty,
\end{equation}
where $\lambda$ is a Lagrange multiplier, and
\[
S(a):=\left\{u\in H^1(\R^3,\R)\quad  \hbox{ s.t. } \quad\|u\|^2_2=a^2\right\}.
\]
A unique positive solution $(\lambda,u_{\alpha})$ to \eqref{eq:g1} is therefore given by
\begin{equation}
\begin{aligned}
\lambda=\frac{\alpha^4a^4}{\|W\|^4_2},\quad  u_{\alpha}=\frac{\lambda}{\alpha}W(\lambda^{\frac{1}{2}}x),
\end{aligned}
\end{equation}
where $W$ is defined in \eqref{eq:w}, and the existence of the latter  is guaranteed  by \cite{KKM}. 
Furthermore,
\begin{equation*}
m_0(a)=-\frac{\alpha^4 a^{6}}{6\|W\|^4_2}.
\end{equation*}
The following Lemma yields the bound away from zero of the minimum $m^+(a_1,a_1)$, which is in turn characterized by  the minimum $m_0$ defined in \eqref{eq:e1}.
\begin{lemma}\label{lem3.5}
Let $\alpha, a_1, a_2>0$. Provided $a=a_1=a_2<D$, then
\begin{equation*}
m^+(a,a)< 3m_0\left(\frac{a}{\sqrt{2}}\right) :=-\frac{\alpha^4 a^{6}}{16\|W\|^4_2}<0.
\end{equation*}
\end{lemma}
\begin{proof}
    For a proof, we refer to \cite[Lemma 3.8]{FLY} by the authors.
\end{proof}
We now now a convergence result linking a rescaled ground state of \eqref{eqA0.2} to  a ground state solution for the functional $E_0$, for the scaling parameter going to zero. 
\begin{lemma}\label{lem3.6}
Let $\alpha, a_1, a_2>0$. Suppose that $\epsilon=a_1=a_2<D$. Then for any ground state ${\bf u}\in S(\epsilon_n,\epsilon_n)$ of \eqref{eqA0.2} with $\{\epsilon_n\}$ a sequence going to zero,
we have, up to a subsequence,
\begin{equation*}
\epsilon_n^{-4}{\bf u}(\epsilon_n^{-2}x):=\left(\epsilon_n^{-4}u_1(\epsilon_n^{-2}x), \epsilon_n^{-4}u_2(\epsilon_n^{-2}x),\epsilon_n^{-4}u_3(\epsilon_n^{-2}x)\right)\to {\bf v}_0
 \quad  \text{in}\quad {\bf H}^{1},
\end{equation*}
where ${\bf v}_0$ is a ground state solution of $E_0$ constrained on $S(1,1)$, and $E_0$ is defined in equations \eqref{energy-limit}.
\end{lemma}
\begin{proof}
Fix $\alpha>0$. For any $\{\epsilon_{n}\}$ with $\epsilon_{n}\to 0^+$ as $n\to +\infty$,  let ${\bf u}_n\in V(\epsilon_{n},\epsilon_{n})$ be a minimizer of $m^+(\epsilon_{n},\epsilon_{n})$, where \[V(\epsilon_{n},\epsilon_{n})=\left\{ {\bf u}_{n}\in S(\epsilon_{n},\epsilon_{n})\cap\mathcal{M} : \|{\bf u}\|_{{\bf {\dot H}}^1}< \rho_0 \right\}.\]
By Lemma \ref{lem2.6}, we get that ${\bf u}_{n}$ is a ground state of $E$ restricted to ${S(\epsilon_{n},\epsilon_{n})}$. Then the Lagrange multipliers rule implies the existence of some $\lambda_{1,\epsilon_{n}},\lambda_{2,\epsilon_{n}} \in \R$ such that
\begin{equation}\label{eq:y4}
\begin{cases}
\displaystyle\int\left( \nabla u_{1,n}\nabla\overline{\phi}_1+ \lambda_{1,\epsilon_n}u_{1,n}\overline{\phi}_1
-|u_{1,n}|^{4}u_{1,n}\overline{\phi}_1\right)=\alpha\mathrm{Re}\int u_{3,n}\overline{u}_{2,n}\overline{\phi}_1, \\
\displaystyle\int \left(\nabla u_{2,n}\nabla \overline{\phi}_2+ \lambda_{2,\epsilon_n}u_{2,n}\overline{\phi}_2
-|u_{2,n}|^{4}u_{2,n}\overline{\phi}_2\right)=\alpha\mathrm{Re}\int u_{3,n}\overline{u}_{1,n}\overline{\phi}_2, \\
\displaystyle\int\left( \nabla u_{3,n}\nabla \overline{\phi}_3+ (\lambda_{1,\epsilon_n}+\lambda_{2,\epsilon_n})u_{3,n}\overline{\phi}_3
-|u_{3,n}|^{4}u_{3,n}\overline{\phi}_3\right)=\alpha\mathrm{Re}\int u_{1,n}u_{2,n}\overline{\phi}_3,\\
\end{cases}
\end{equation}
for each ${\bm\phi}=(\phi_1,\phi_2,\phi_3)\in {\bf H}^{1}$.
\vskip3mm
We claim that 
\begin{equation}\label{eq:Dd2.2}
\frac{\alpha^4 \epsilon_{n}^{4}}{8\|W\|^4_2}<\lambda_{1,\epsilon_{n}}+\lambda_{2,\epsilon_{n}}<\frac{81\alpha^4 \epsilon_{n}^{4}}{8\|W\|^4_2}.
\end{equation}
By using twice that $P({\bf u}_{n})=0$ and by means Lemma \ref{lem3.5}, we have:
\begin{equation}\label{eq:x3}
 -\frac{\alpha^4 \epsilon_{n}^{6}}{16\|W\|^4_2}>E({\bf u}_{n})=\frac{1}{3}\|{\bf u}_n\|_{{\bf {\dot H}}^1}^2-\frac{3\alpha}{4}\mathrm{Re}\int u_{1,n}u_{2,n}\overline{u}_{3,n}
 \end{equation}
 and
 \begin{equation}\label{eq:x3bis}
 -\frac{\alpha^4 \epsilon_{n}^{6}}{16\|W\|^4_2}>E({\bf u}_{n})=-\frac{1}{6}\|{\bf u}_n\|_{{\bf {\dot H}}^1}^2+\frac{1}{2}\| {\bf u}_n\|^6_6.
\end{equation}
It follows immediately from \eqref{eq:x3bis} that $\displaystyle \frac{3\alpha^4\epsilon_{n}^{6}}{8\|W\|^4_2}< \|{\bf u}_n\|_{{\bf {\dot H}}^1}^2$; while, since the left-hand side of \eqref{eq:x3} is negative, by using \eqref{eq:z4}
we get 
\begin{equation}\label{eq:ba1bis}
\frac{1}{3}\|{\bf u}_n\|_{{\bf {\dot H}}^1}^2< \frac{3\alpha}{4}\mathrm{Re}\int u_{1,n}u_{2,n}\overline{u}_{3,n}< \frac{3A_2}{4}\|{\bf u}_n\|_{{\bf {\dot H}}^1}^{\frac32}
\end{equation}
and then $\displaystyle \|{\bf u}_n\|_{{\bf {\dot H}}^1}^2 < \frac{243\alpha^4\epsilon_n^6}{16\|W\|_2^4}$.
In conclusion
\begin{equation*}\label{ba1}
  \frac{3\alpha^4\epsilon_{n}^{6}}{8\|W\|^4_2} < \|{\bf u}_n\|_{{\bf {\dot H}}^1}^2 < \frac{243\alpha^4\epsilon_n^6}{16\|W\|_2^4}.
\end{equation*}
It follows from \eqref{eq:y4} that
\begin{equation*}
\begin{aligned}
(\lambda_{1,\epsilon_{n}}+\lambda_{2,\epsilon_{n}})\epsilon_{n}^2=-\|{\bf u}_n\|_{{\bf {\dot H}}^1}^2+\|{\bf u}_n\|_{6}^6+3\alpha \mathrm{Re}\int u_{1,n}u_{2,n}\overline{u}_{3,n}.
\end{aligned}
\end{equation*}
Exploiting again the fact that $P({\bf u}_n)=0$ and by mimicking the same estimates as in \eqref{eq:ba1bis},  we write 
\begin{equation}\label{eq:ba1quat}
(\lambda_{1,\epsilon_{n}}+\lambda_{2,\epsilon_{n}})\epsilon_{n}^2=\frac{3}{2}\alpha\mathrm{Re}\int u_{1,n}u_{2,n}\overline{u}_{3,n}<\frac{81\alpha^4\epsilon_n^6}{8\|W\|_2^4}.
\end{equation}
By using instead the definition of the energy and Lemma \ref{lem3.5} , it is easy too see that 
\begin{equation}\label{eq:ba1ter}
(\lambda_{1,\epsilon_{n}}+\lambda_{2,\epsilon_{n}})\epsilon_{n}^2=-2E({\bf u}_n)>\frac{\alpha^4\epsilon_{n}^{6}}{8\|W\|^4_2}
\end{equation}
Estimates \eqref{eq:ba1ter} and \eqref{eq:ba1quat} gives \eqref{eq:Dd2.2} and the proof of the claim is done.

\vskip2mm

Define now 
\begin{equation}\label{def:v-i}
{\bf v}_n=\epsilon_{n}^{-4}{\bf u}_{n}(\epsilon_{n}^{-2}x).
\end{equation}
Then, for $j\in\{1,2,3\}$,
\begin{equation*}
\|\nabla v_{j,n}\|^2_2=\epsilon_{n}^{-6}\|\nabla u_{j,n}\|^2_2, \quad \|v_{j,n}\|^6_6=\epsilon_{n}^{-18}\|u_{j,n}\|^6_6, \quad \text{and} \quad \|v_{j,n}\|^2_2=\epsilon_{n}^{-2}\|u_{j,n}\|^2_2.
\end{equation*}
Therefore, for $\epsilon_{n}\to 0$ as $n\to \infty$, we have
\begin{equation*}
\begin{aligned}
m^+(\epsilon_{n},\epsilon_{n})+o_n(1)&= E({\bf u}_n)= \epsilon_{n}^{6}E_0({\bf v}_n)- \epsilon_{n}^{18}\|{\bf v}_n\|^6_6\\
&\ge \epsilon_{n}^{6}m_0(1,1)+o\left(\epsilon_{n}^{6}\right).
\end{aligned}
\end{equation*}
\vskip1mm
\noindent From the definition of $m_0(1,1)$, see \eqref{eq:e1}, for any $\varepsilon>0$, there exists ${\bf v}_0\in S(1,1)$ such that
\begin{equation*}
E_0({\bf v}_0)\leq m_0(1,1)+\varepsilon.
\end{equation*}
By definition \eqref{def:v-i}, $u_{j,\epsilon_{n}}:=\epsilon_{n}^4v_{j,0}\left(\epsilon_{n}^2 x\right)$ for $j\in\{1,2,3\}$. Therefore, ${\bf u}_{\epsilon_{n}}\in V(\epsilon_{n},\epsilon_{n})$ for $\epsilon_{n}$ small enough. Then
\begin{equation*}\label{x21}
\begin{aligned}
m^+(\epsilon_{n},\epsilon_{n})=\inf_{{\bf u} \in V(\epsilon_{n},\epsilon_{n})} E({\bf u})&\leq E({\bf u}_{\epsilon_{n}})\leq \epsilon_{n}^{6}E_0({\bf v}_0)+ \epsilon_n^{18}\|{\bf v}_0\|^6_6\\
&\leq\epsilon_{n}^{6}\left(m_0(1,1)+\varepsilon\right)+o\left(\epsilon_{n}^{6}\right),
\end{aligned}
\end{equation*}
for all $\varepsilon>0$ and $\epsilon_{n}>0$ small enough. Therefore,
\[
m^+(\epsilon_{n},\epsilon_{n})= \epsilon_{n}^{6}m_0(1,1)+o\left(\epsilon_{n}^{6}\right).
\]
This implies that $\{{\bf v}_n\}$ is a minimizing sequence for $m_0(1,1)$. If $\{{\bf u}_n\}$ is a minimizing sequence of $m^+(\epsilon_{n},\epsilon_{n})$,  $E({\bf u}_n)=m^+(\epsilon_{n},\epsilon_{n})+o(1)$.
By the definition of $\{{\bf v}_n\}$, see \eqref{def:v-i}, we have
\[
E({\bf v}_n)=E(\epsilon_{n}^{-4}{\bf u}_n(\epsilon_{n}^{-2}x))=m_0(1,1)+o(\epsilon_{n}^6),
\]
i.e., $\{{\bf v}_n\}$ is a minimizing sequence of $m_0(1,1)$.
Up to a subsequence, there exists a radially symmetric Palais-Smale sequence $\{{\tilde{\bf v}}_n\}$ such that $\|{\tilde{\bf v}}_n-{\bf v}_n\|_{{\bf H}^1}=o_n(1)$.
Similar to the proof of Lemma 3.6 in \cite{FLY}, up to translation, there exists a minimizer ${\bf v}_0$ for $m_0(1,1)$ such that ${\tilde{\bf v}}_n\to {\bf v}_0$ in ${\bf H}^1$.
Indeed, by \cite[Lemma 3.6]{FLY} for any minimizing sequence of $m_0(1,1)$, there exists a compact subsequence.
\end{proof}

We give now an essential local compactness result for the functional E({\bf v}).

\begin{lemma}\label{lem3.3}
Let $\alpha, a_1, a_2>0$. If $\max\{a_1,a_2\}<D$, then $m^-(a_1,a_2)$ is achieved by a function in $(a_1,a_2)$, which is real-valued, positive and radially symmetric.
\end{lemma}

\begin{proof}
We need to show that $m^-_r(a_1,a_2)$ is attained.
Therefore, we can choose a real (component-wise) non-negative and radially symmetric Palais-Smale sequence $\{{\bf u}_n\}$ for $m^-(a_1,a_2)$ with $P({\bf u}_{n})=o_n(1)$, i.e. $\lim\limits_{n\to \infty} E({\bf u}_{n})=m^-(a_1,a_2)$ and $E'|_{S(a_1,a_2)}({\bf u}_n)\to 0$ as $n\to \infty$ (see \cite{FLY,Soave2}).
Similar to the proof of Lemma \ref{lem2.6}, we have that the sequence $\{{\bf u}_n\}$ is bounded in $H^1(\R^3,\R^3)$.
There exists $\tilde{\bf u}$ such that ${\bf u}_{n}\rightharpoonup \tilde{\bf u}$ in ${\bf H}^1$. Hence, the limit ${\tilde{\bf u}}$ satisfies $\tilde{u}_j\ge 0$, for $j\in\{1,2,3\}$, the latter being  radial functions.\\

\noindent   In order to prove the strong convergence in ${\bf H}^1$, we firstly claim that the following crucial refined bound:
\begin{equation}\label{eq:A1b}
m^-(a_1,a_2)<m^+(a_1,a_2)+\frac{1}{3}C_{\Sob}^{\frac{3}{2}}.
\end{equation}
Let ${\bf u}=(u_1,u_2,u_3)\in S(a_1,a_2)$ be the ground state solution with $\lambda_1,\lambda_2>0$, i.e., $E({\bf u})=m^+(a_1,a_2)$, and ${\bf u}$ is a smooth solution   to \eqref{eq:b4}.
Recall from \eqref{Sobolev} the definition of $U_\varepsilon$. Let $\tilde{U}_{\varepsilon}=\chi(x)U_{\varepsilon}$ where $\chi$ is a cut-off function such that $\chi(x)=1$ for $|x|\leq1$ and $\chi(x)=0$ for $|x|>2$, we have
\begin{equation}\label{eq:3.4u}
\|\nabla \tilde{U}_{\varepsilon}\|^2_2=C_{\Sob}^{\frac{3}{2}}+O(\varepsilon)\quad  \text{and} \quad \|\tilde{U}_{\varepsilon}\|^{6}_{6}=C_{\Sob}^{\frac{3}{2}}+O(\varepsilon^2).
\end{equation}
We define 
\[
w_{\varepsilon,t}=u_1+t\tilde{U}_{\varepsilon}
\quad\hbox{ and }\quad \tilde{w}_{\varepsilon,t}=s^{\frac{1}{2}}w_{\varepsilon,t}(sx).
\]
By direct calculations, we have
\begin{equation*}
\|\nabla \tilde{w}_{\varepsilon,t}\|^2_2=\|\nabla w_{\varepsilon,t}\|^2_2, \qquad \|\tilde{w}_{\varepsilon,t}\|^{6}_{6}=\| w_{\varepsilon,t}\|^{6}_{6},
\end{equation*}
and
\[
\|\tilde{w}_{\varepsilon,t}\|^2_2=s^{-2}\|w_{\varepsilon,t}\|^2_2.
\]
We choose $s=\frac{\|w_{\varepsilon,t}\|_2}{\|u_1\|_2}$ such that $(\tilde{w}_{\varepsilon,t},u_2,u_3)\in S(a_1,a_2)$. By Lemma \ref{lem2.3}, there exists $\tau_{\varepsilon,t}\in \R$ such that $\tau_{\varepsilon,t}\star \left(\tilde{w}_{\varepsilon,t},u_2,u_3\right)\in \mathcal{P}^{-}_{a_1,a_2}$. Thus
\begin{equation}\label{eq:3.5b}
\tau^{\frac{1}{2}}_{\varepsilon,t}\left(\sum^{3}_{j=2}\|\nabla u_j\|^2_2+\|\nabla \tilde{w}_{\varepsilon,t}\|^2_2\right)= \tau^{\frac{9}{2}}_{\varepsilon,t}\left(\sum^{3}_{j=2}\|u_j\|^6_6
+\|\tilde{w}_{\varepsilon,t}\|^{6}_{6}\right)+\alpha \int\tilde{w}_{\varepsilon,t}u_2u_3.
\end{equation}
Since ${\bf u}\in \mathcal{P}^+_{a_1,a_2}$, from Lemma \ref{lem2.6}, $\tau_{\varepsilon,0}>0$. From \eqref{eq:3.4u} and \eqref{eq:3.5b}, we have $\tau_{\varepsilon,t}\to 0$ as $t\to +\infty$ and $\varepsilon>0$ small enough.  
Let us observe that $m^-(a_1,a_2)\le E(\tau_{\varepsilon,t}\star \tilde{w}_{\varepsilon,t},u_2,u_3)$. As there  
exists a $t_{\varepsilon}$ such that $\tau_{\varepsilon,t_{\varepsilon}}=1$ for $\varepsilon$ small enough, we can consider  
\begin{equation*}
 m^-(a_1,a_2)\le \sup_{t\ge 0}E\left(\tilde{w}_{\varepsilon,t},u_2,u_3\right).
\end{equation*}

\noindent Moreover, there exists $t_0>0$ such that
\begin{equation*}
\begin{aligned}
E\left(\tilde{w}_{\varepsilon,t},u_2,u_3\right)&<m^+(a_1,a_2)+\frac{1}{3}C_{\Sob}^{\frac{3}{2}}
\end{aligned}
\end{equation*}
for $\frac{1}{t_0}\leq t \leq t_0$. 
Since the function $\tilde{U}_\varepsilon$ is compactly supported, we have that
\begin{equation*}
\int u_1\tilde{U}_{\varepsilon} \sim  \varepsilon^{\frac{5}{2}}\int^{\frac{1}{\varepsilon}}_{1}\frac{1}{(1+r^2)^{1/2}}r^{2} dr
\sim \varepsilon^{\frac{1}{2}} 
\end{equation*}
and 
\[\int|\tilde{U}_\varepsilon|^2\sim \varepsilon,
\]
Thus, by the definition of $s$ and $w_{\varepsilon,t}$,
\begin{equation}\label{eq:3.6c}
s^2=\frac{\|w_{\varepsilon,t}\|^2_2}{\|u_1\|^2_2}=1+\frac{2t}{\|u_1\|^2_2}\int u_1\tilde{U}_{\varepsilon}
+\frac{t^2\|\tilde{U}_{\varepsilon}\|^2_2}{\|u_1\|^2_2}=1+O(\varepsilon^{\frac{1}{2}})
\end{equation}
for $\frac{1}{t_0}\leq t\leq t_0$. In addition, since for $a,b>0$ $(a+b)^{6}\geq a^{6}+b^{6}+6(a^{5}b+ab^{5})$ we get
\begin{equation*}
\begin{aligned}
E\left(\tilde{w}_{\varepsilon,t},u_2,u_3\right)&=\frac{1}{2}\left(\int|\nabla \tilde{w}_{\varepsilon,t}|^2+\sum_{j=2}^3\int|\nabla u_j|^2
\right)-\frac{1}{6}\left(\int|\tilde{w}_{\varepsilon,t}|^{6}+\sum_{j=2}^3\int|u_j|^6\right)\\
&\qquad-\alpha\int\tilde{w}_{\varepsilon,t}u_2u_3\\
&\leq E({\bf u})+\frac{t^2}{2}\int|\nabla \tilde{U}_\varepsilon|^2
-\frac{t^{6}}{6}\int|\tilde{U}_\varepsilon|^{6}
\\
&\qquad+t\int \nabla u_1\nabla \tilde{U}_\varepsilon-\int u_1|t\tilde{U}_\varepsilon|^{5}-t\int|u_1|^{4}u_1\tilde{U}_\varepsilon
\\
&\qquad-\alpha t\int u_2u_3s^{\frac{1}{2}}\tilde{U}_{\varepsilon}(sx)-\alpha\left(\int s^{\frac{1}{2}}u_1(sx)u_2u_3-\int u_1u_2u_3\right).
 \end{aligned}
\end{equation*}
We conclude from \eqref{eq:3.6c} and the Taylor expansion  in $s$ centered in $s_0=1$ that
\begin{equation}\label{eq:AL3}
\begin{aligned}
\int s^{\frac{1}{2}}u_1(sx)u_2u_3
= \int u_1u_2u_3+\frac{t}{\|u_1\|^2_2}\int u_1\tilde{U}_{\varepsilon}\int \left(\frac{1}{2}u_1u_2u_3+u_2u_3\nabla u_1\cdot x\right)+o(\varepsilon^{\frac{1}{2}}).
\end{aligned}
\end{equation}
Since ${\bf u}$ solves \eqref{eq:b4}, by multiplying both sides of the first equation of \eqref{eq:b4} by $x\cdot \nabla u_1$ and $u_1$, integrating over $\R^3$, we get
\begin{equation}\label{eq:AL}
\alpha\int u_2u_3\nabla u_1\cdot x=-\frac{1}{2}\int|\nabla u_1|^2+\frac{1}{2}\int|u_1|^{6}-\frac{3\lambda_1}{2}\int |u_1|^{2},
\end{equation}
and
\begin{equation}\label{eq:AL1}
\alpha\int u_2u_3u_1=\int|\nabla u_1|^2
-\int|u_1|^{6}+\lambda_1\int|u_1|^{2},
\end{equation}
respectively.
It follows from \eqref{eq:AL} and \eqref{eq:AL1} that
\begin{equation}\label{eq:AL3-b}
\frac{t}{\|u_1\|^2_2}\int u_1\tilde{U}_{\varepsilon}\int \left(\frac{1}{2}u_1u_2u_3+u_2u_3\nabla u_1\cdot x\right)=-\frac{\lambda_1t}{\alpha}\int u_1\tilde{U}_{\varepsilon}.
\end{equation}
Therefore, we conclude from \eqref{eq:AL3} and \eqref{eq:AL3-b} that
\begin{equation}\label{eq:ide3.22}
\int s^{\frac{1}{2}}u_1(sx)u_2u_3-\int u_1u_2u_3
=-\frac{\lambda_1t}{\alpha}\int u_1\tilde{U}_{\varepsilon}+o(\varepsilon^{\frac{1}{2}}).
\end{equation}
Similarly, by multiplying both sides of the first equation of \eqref{eq:b4} by $t\tilde{U}_{\varepsilon}$, integrating over $\R^3$, we have
\begin{equation}\label{eq:ide3.22b}
t\int\nabla u_1\nabla \tilde{U}_{\varepsilon}+\lambda_1t\int u_1\tilde{U}_{\varepsilon}-t\int|u_1|^{4}u_1\tilde{U}_\varepsilon
=\alpha t\int u_2u_3\tilde{U}_{\varepsilon}.
\end{equation}
Then, for $\varepsilon$ sufficiently small, by means of \eqref{eq:ide3.22} and \eqref{eq:ide3.22b} we have that 
\begin{equation*}
\begin{aligned}
\mathcal{R}_{\varepsilon}:&=t\int \nabla u_1\nabla \tilde{U}_\varepsilon-t\int |u_1|^{4}u_1\tilde{U}_\varepsilon
-\alpha t\int u_2u_3s^{\frac{1}{2}}\tilde{U}_{\varepsilon}(sx)\\
&\quad-\alpha\left(\int s^{\frac{1}{2}}u_1(sx)u_2u_3
-\int u_1u_2u_3\right)-\int u_1|t\tilde{U}_\varepsilon|^{5}\\
&=t\int\nabla u_1\nabla \tilde{U}_\varepsilon-t\int|u_1|^{4}u_1\tilde{U}_\varepsilon
-\alpha t\int u_2u_3s^{\frac{1}{2}}\tilde{U}_{\varepsilon}(sx)
\\
&\qquad +\lambda_1t\int u_1\tilde{U}_{\varepsilon} -\int u_1|t\tilde{U}_\varepsilon|^{5}+o(\varepsilon^{\frac{1}{2}})\\
&=\alpha t\int u_2u_3\tilde{U}_{\varepsilon}-\alpha t\int u_2u_3s^{\frac{1}{2}}\tilde{U}_{\varepsilon}(sx)+o(\varepsilon^{\frac{1}{2}})
-\int u_1|t\tilde{U}_\varepsilon|^{5}.
\end{aligned}
\end{equation*}
At this point  we expand again in Taylor, and by means of the estimate
\begin{equation*}
\int u_1|\tilde{U}_{\varepsilon}|^{5}\sim \varepsilon^{\frac{1}{2}}\int^{\frac{1}{\varepsilon}}_{1} \frac{r^2}{(1+r^2)^{\frac52}} dr
\sim \varepsilon^{\frac{1}{2}}
\end{equation*}
we write 
\begin{equation*} 
\begin{aligned}
\mathcal R_\varepsilon&\sim \left(1-s\right)\alpha t\int u_2u_3(\nabla \tilde{U}_{\varepsilon}\cdot x +\frac12\tilde {U}
_\varepsilon)+o(\varepsilon^{\frac{1}{2}})-\varepsilon^{\frac{1}{2}}\\
&=\left(1-s\right)\alpha t\int u_2u_3\left((\nabla\chi\cdot x)U_{\varepsilon}+\chi(\nabla U_{\varepsilon}\cdot x) +\frac12\tilde {U}
_\varepsilon \right)+o(\varepsilon^{\frac{1}{2}})-\varepsilon^{\frac{1}{2}}\\
&\sim \varepsilon+o(\varepsilon^{\frac{1}{2}})-\varepsilon^{\frac{1}{2}}<0,
\end{aligned}
\end{equation*}
where in the last step we used that 
\begin{equation*}
\left(1-s\right)\alpha t\int u_2u_3\left((\nabla\chi\cdot x)U_{\varepsilon}+\chi(\nabla U_{\varepsilon}\cdot x)+\frac12\tilde {U}
_\varepsilon \right)\sim \varepsilon.
\end{equation*}
Thus, by gluing all the estimates above together, we conclude with
\begin{equation*}
\begin{aligned}
E\left(\tilde{w}_{\varepsilon,t},u_2,u_3\right)&\leq E({\bf u})+\frac{t^2}{2}\int |\nabla \tilde{U}_\varepsilon|^2
-\frac{t^{6}}{6}\int|\tilde{U}_\varepsilon|^{6}+\mathcal{R}_{\varepsilon}\\
&< m^+(a_1,a_2)+\frac{1}{3}C_{\Sob}^{\frac{3}{2}}
\end{aligned}
\end{equation*}
for $\frac{1}{t_0}\leq t\leq t_0$. 
Therefore, we get
\begin{equation*}
m^-(a_1,a_2)\leq\sup_{t\ge 0}
E\left(\tilde{w}_{\varepsilon,t},u_2,u_3\right)<m^+(a_1,a_2)+\frac{1}{3}C_{\Sob}^{\frac{3}{2}}.
\end{equation*}

\vskip1mm

\noindent   At this point we can prove the strong convergence by distinguishing two cases. By an argument analogous to the one used in Lemma \ref{lem2.6}, we prove that for any $j\in\{1,2,3\}$, $\tilde u_j\neq0$.
\vskip1mm
Suppose that there exist a $j\in\{1,2,3\}$ such that $\tilde{u}_j=0$; then
\[\|\nabla{\bf u}_{n}\|_2^2 -\|{\bf u}_{n}\|_6^6=o_n(1).
\]
It follows that
\[\lim\limits_{n\to \infty}\|\nabla{\bf u}_{n}\|_2^2\ge C_{\Sob}^{\frac{3}{2}}.\]
Hence
\begin{equation*}
\begin{aligned}
m^-(a_1,a_2)&=E({\bf u}_n)-\frac{1}{6}P({\bf u}_n)+o_n(1)\\
&=\frac{1}{3}\|\nabla{\bf u}_{n}\|_2^2-\frac{3\alpha}{4}\int u_{1,n}u_{2,n}u_{3,n}+o_n(1)\\
&\geq\frac13 \|\nabla{\bf u}_{n}\|_2^2 \ge \frac{1}{3}C_{\Sob}^{\frac{3}{2}}.
\end{aligned}
\end{equation*}
As from \eqref{eq:A1b} we have that $m^-(a_1,a_2)<\frac{1}{3}C_{\Sob}^{\frac{3}{2}}$, this is a contradiction and hence $\tilde u_j=0$ cannot hold.
\vskip1mm

\noindent Therefore $\tilde{u}_j\not\equiv 0$ for any $j \in \{1,2,3\}$, and it follows from the maximum principle (see \cite[Theorem 2.10]{hanq}) that $\tilde{u}_j>0$. By  \cite[Lemma A.2]{IN}, we have $\lambda_1,\lambda_2\!>\!0$. Let ${\bf v}:=(v_{1,n},v_{2,n},v_{3,n})=(u_{1,n}-\tilde{u}_1,u_{2,n}-\tilde{u}_2,u_{3,n}-\tilde{u}_3)$.
We can apply an analysis similar to that in the proof of Lemma \ref{lem2.6} and show that
$\|\tilde{u}_1\|^2_2+\|\tilde{u}_2\|^2_2=a^2_1$ and $\|\tilde{u}_2\|^2_2+\|\tilde{u}_3\|^2_2=a^2_2$, and hence ${\tilde{\bf u}} \in P_{a_1,a_2}$.
Then, we distinguish  two sub-cases: either
\[
(i)\ \ \|\nabla  v_{j,n}\|_2^2\to 0\ \text{for~any} \ j\in \{1,2,3\},
\]
or 
\[(ii)\ \ \  \|\nabla v_{j,n}\|_2^2\to \ell>0 \ \text{for~at~least~one} \ j\in\{1,2,3\}.
\]
If $(ii)$ holds, we have
\begin{equation*}
\begin{aligned}
m^-(a_1,a_2)&=E({\bf u}_n)+o_n(1)\\
&=\frac{1}{2}\|\nabla{\tilde{\bf u}}\|_2^2-\frac{1}{6}\|{\tilde{\bf u}}\|_6^6-\alpha\int \tilde{u}_{1}\tilde{u}_{2}\tilde{u}_{3}\\
&+\frac{1}{2}\|\nabla{\bf v}_{n}\|_2^2-\frac{1}{6}\|{\bf v}_{n}\|_6^6+o_n(1)\\
 &\geq \frac{1}{3}C_{\Sob}^{\frac{3}{2}}+m^+(a_1,a_2)+o_n(1),
\end{aligned}
\end{equation*}
which is a contradiction with respect to \eqref{eq:A1b}. Eventually,  case $(i)$ holds, which implies that ${\bf u}_n\to{\tilde{\bf u}}$ in ${\bf H}^1$. We then conclude that $E(\tilde{{\bf u}})=m^-_r(a_1,a_2)$. 
\end{proof}

\begin{lemma}\label{th1.5}
Under the assumptions of Lemma \ref{lem3.3}, let $(\tilde{u}_1,\tilde{u}_2,\tilde{u}_3)\in S(a_1,a_2)$ be a positive radial solution of \eqref{eqA0.2}. As $(a_1,a_2)\to (0^+,0^+)$,  $m^-(a_1,a_2)\to \frac{1}{3}C_{\Sob}^{\frac{3}{2}}$ and $\|\nabla{  \tilde{\bf u}}\|_{2}^{2} \to C_{\Sob}^{\frac{3}{2}}$.
\end{lemma}

\begin{proof}
By Lemma \ref{lem3.3}, there exists an excited state solution $\tilde{\bf u}=(\tilde{u}_1,\tilde{u}_2,\tilde{u}_3)\in \mathcal{P}^-_{a_1,a_2}$ which satisfies, see \eqref{eq:A1b},
\[E(\tilde{\bf u})<m^+(a_1,a_2)+\frac{1}{3}C_{\Sob}^{\frac{3}{2}}.\]
Firstly, we prove that
\begin{equation}\label{eq:c15}
E(\tilde{\bf u}) \to \frac{1}{3}C_{\Sob}^{\frac{3}{2}}\quad \text{as} ~(a_1,a_2) \to (0^+,0^+).
\end{equation}
By $P(\tilde{\bf u}) = 0$,
\begin{equation*}\label{n1}
\begin{aligned}
E(\tilde{\bf u})&=\frac{1}{2}\|\nabla \tilde{\bf u}\|_2^2-\frac{1}{6}\|\tilde{\bf u}\|_6^6-\alpha\int \tilde{u}_1\tilde{u}_2{\tilde{u}}_3\\
&= \frac{1}{3}\|\nabla \tilde{\bf u}\|_2^2-\frac{3\alpha}{4} \int \tilde{u}_1\tilde{u}_2{\tilde{u}}_3.
\end{aligned}
\end{equation*}
Since
\begin{equation*}
E(\tilde{\bf u}) < m^+(a_1,a_2) +  \frac{1}{3}C_{\Sob}^{\frac{3}{2}} \leq \frac{1}{3}C_{\Sob}^{\frac{3}{2}},
\end{equation*}
we deduce that $\{\tilde{\bf u}_n\} \subset {\bf H}^1$ is uniformly bounded.
From $P(\tilde{\bf u})=0$, we get $\displaystyle \lim_{(a_1,a_2) \to (0^+,0^+)} \|\nabla\tilde{\bf u}\|_2^2= \lim_{(a_1,a_2) \to (0^+,0^+)}\|\tilde{\bf u}\|^6_6 $, because $\displaystyle \int u_1u_2u_3\to 0$ as a consequence of the Gagliardo-Nirenberg's inequality and the fact that the masses go to zero jointly with the uniform boundedness of the $H^1$-norm. Hence
\begin{align*}
\ell := \lim_{n\to \infty} \|\nabla\tilde{\bf u}_n\|_2^2= \lim_{n\to \infty}\|\tilde{\bf u}_{n}\|^6_6 \leq{C_{\Sob}^{-3}} \lim_{n\to \infty} \|\nabla \tilde{\bf u}_n\|_2^6= {C_{\Sob}^{-3}} \ell^{3}.
\end{align*}
Therefore, either $\ell = 0$ or $\ell \geq C_{\Sob}^{\frac{3}{2}}$. We claim that $\ell = 0$  is impossible. Indeed, since  $\tilde{\bf u} \in \mathcal{P}^{-}_{a_1,a_2}$, we have
\begin{align*}
\|\nabla\tilde{\bf u}_n\|_2^2 <9 \|\tilde{\bf u}\|_{6}^{6}
\leq 9C_{\Sob}^{-3}\|\nabla \tilde{\bf u}_n\|_2^6.
\end{align*}
Therefore, $\ell\ge C_{\Sob}^{\frac{3}{2}}$ and as $P(\tilde{\bf u})=0$, we have
\begin{align*}
E(\tilde{\bf u}_n)&=\frac{1}{2}\|\nabla \tilde{\bf u}_{n}\|_{2}^{2} -\frac{1}{6}\| {\bf\tilde u}_{n}\|^6_6-\alpha\mathrm{Re}\int \tilde{u}_{1,n} \tilde{u}_{2,n} \overline{\tilde{u}}_{3,n}\\
&\ge  \frac{1}{3}\|\nabla \tilde{\bf u}_{n}\|_{2}^{2} + o_n(1)
\geq \frac{1}{3}C_{\Sob}^{\frac{3}{2}}+o_n(1).
\end{align*}
Moreover, $m^+(a_1,a_2) \to 0$ as $(a_1,a_2)\to (0^+,0^+)$, and we see that
\begin{align*}
E(\tilde{\bf u}_n)= m^-(a_1,a_2) < m^+(a_1,a_2) + \frac{1}{3}C_{\Sob}^{\frac{3}{2}}.
\end{align*}
We obtain \eqref{eq:c15}, which implies that
\[
\|\nabla\tilde{\bf u}_{n}\|_{2}^{2} \to C_{\Sob}^{\frac{3}{2}} \quad \text{ as } \quad (a_1,a_2)\to (0^+,0^+).\]
Then we have
\begin{equation*}
m^-(a_1,a_2)\to \frac{1}{3}C_{\Sob}^{\frac{3}{2}} \quad  \text{as} \quad  (a_1,a_2)\to (0^+,0^+).
\end{equation*}
\end{proof}
At this point we can combine the results above to prove Theorem \ref{th1.1}.

\begin{proof}[Proof of Theorem \ref{th1.1}]
The first point follows from Lemmas \ref{lem2.6} and \ref{lem3.3}. The other points are derived by Lemmas \ref{lem3.6} and \ref{th1.5}.
\end{proof}

\section{Dynamical results}

This section is devoted to the dynamical results, in particular we will prove Theorems \ref{th1.2} and \ref{th1.3}.
\\

We go back to the original time-dependent Cauchy problem \eqref{eqA0.1}, namely
\begin{equation}\label{EqA1}
\begin{cases}
i \partial_{t} \psi_{1}=-\Delta \psi_{1}-\left|\psi_{1}\right|^{4} \psi_{1}-\alpha \psi_{3} \overline{\psi}_{2}, \\
i \partial_{t} \psi_{2}=-\Delta \psi_{2}-\left|\psi_{2}\right|^{4} \psi_{2}-\alpha \psi_{3} \overline{\psi}_{1}, \\
i \partial_{t} \psi_{3}=-\Delta \psi_{3}-\left|\psi_{3}\right|^{4} \psi_{3}-\alpha \psi_{1} \psi_{2}.
\end{cases}
\end{equation}
with initial datum ${\bm \psi}(0,x)=(\psi_1(0,x),\psi_2(0,x),\psi_3(0,x))={\bm \psi}_0(x)\in {\bf H}^1$, and we recall the definition of the set of ground states
\begin{equation*}
\mathcal{G}=\left\{(e^{i\theta_1}u_1,e^{i\theta_2}u_2,e^{i(\theta_1+\theta_2)}u_3) \hbox{ s.t. }\theta_1,\theta_2\in \R, \ {\bf u}\in S(a_1,a_2), \ E({\bf u})=m^+(a_1,a_2)\right\}.
\end{equation*}
Our aim is to prove a uniform local well-posedness result, and then extend local solutions globally in time. This section is  inspired by the recent work of Jeanjean, Jendrej, Le, and Visciglia, see \cite{JJL}.
 \vskip1mm

In order to proceed, we recall the notion of integral equation associated with \eqref{EqA1}.
We first introduce the Strichartz spaces as the Bochner spaces of functions $f :[0,T]\to  L^r$ or   $f :[0,T]\to W^{1,r}$ endowed with the norm 
\begin{equation*}
\|f\|_{Y_{p,r,T}}:=\left(\int^{T}_{0}\|f(t,\cdot)\|^{p}_{r} dt\right)^{\frac{1}{p}},
\end{equation*}
and
\begin{equation*}
\|f\|_{X_{p,r,T}}:=\left(\int^{T}_{0}\|f(t,\cdot)\|^{p}_{W^{1,r}} dt\right)^{\frac{1}{p}},
\end{equation*}
respectively. Here, $(p,r)$ is an admissible Strichartz pair, i.e., it satisfies the scaling relation $\frac{2}{p}+\frac{3}{r}=\frac{3}{2}$, with  $p, r \in [2,\infty]$.
We define the spaces
\[Y_{T}:=Y_{p_1,r_1,T}\cap Y_{p_2,r_2,T}\]
and
\[
X_{T}:=X_{p_1,r_1,T}\cap X_{p_2,r_2,T}\]
where the pairs $(p_j,r_j)$, $j=\{1,2\}$, are given by
\begin{equation}\label{str-pairs}
(p_1,r_1)=\left(12,\frac{9}{4}\right) \quad \hbox{ and } \quad (p_2,r_2)=\left(6,\frac{18}{7}\right).
\end{equation}
The latter are two specific admissible pairs which will be extensively used later on.
 The spaces $X_T$ and $Y_T$ are equipped with the  norms
\begin{align*}
\|w\|_{Y_{T}}&=\|w\|_{Y_{p_1,r_1,T}}+\|w\|_{Y_{p_2,r_2,T}}, \\
\|w\|_{X_{T}}&=\|w\|_{X_{p_1,r_1,T}}+\|w\|_{X_{p_2,r_2,T}},
\end{align*}
and they naturally extend to vector functions by defining
\begin{align*}
\|{\bf w}\|_{{\bf Y}_{T}}&=\sum_{j=1}^3\|w_j\|_{Y_{T}},\\
\|{\bf w}\|_{{\bf X}_{T}}&=\sum_{j=1}^3\|w_j\|_{X_{T}}.
\end{align*}
and any of the functions  $w_j(t,x)$ is defined on the space-time strip $[0,T]\times \R^3$. 
\vskip1mm

\begin{definition} \label{def1}
Let $T>0$. We say that ${\bm \psi}(t,x)=\left(\psi_{1}(t,x),\psi_{2}(t,x),\psi_{3}(t,x)\right)$ is an integral solution of the Cauchy problem \eqref{EqA1} on the time interval $[0,T]$ if:\smallskip

\noindent $\textup{(i)}$ ${\bm \psi}\in C([0,T], {\bf H}^1)\cap {\bf X}_{T}$;
\smallskip

\noindent $\textup{(ii)}$  for all $t\in (0,T]$, it holds that
\begin{equation}\label{eqq:4.1}
\begin{cases}
\displaystyle \psi_{1}(t)=e^{it\Delta}\psi_{0,1}+i\int^{t}_{0}e^{i(t-s)\Delta}g_1({\bm \psi}(s))ds, \\
\displaystyle \psi_{2}(t)=e^{it\Delta}\psi_{0,2}+i\int^{t}_{0}e^{i(t-s)\Delta}g_2({\bm \psi}(s))ds, \\
\displaystyle \psi_{3}(t)=e^{it\Delta}\psi_{0,3}+i\int^{t}_{0}e^{i(t-s)\Delta}g_3({\bm \psi}(s))ds, \\
\end{cases}
\end{equation}
where
\begin{equation}
    \begin{aligned}
     g_1({\bm \psi})&:=|\psi_{1}|^{4}\psi_{1}+\alpha \psi_{3} \overline{\psi}_{2},\\
     g_2({\bm \psi})&:=|\psi_{2}|^{4} \psi_{2}+\alpha \psi_{3} \overline{\psi}_{1},\\
     g_3({\bm \psi})&:=|\psi_{3}|^{4} \psi_{3}+\alpha \psi_{1}\psi_2.   
    \end{aligned}
\end{equation}
\end{definition}

\noindent It follows from \cite[Lemma 3.6]{JJL} that $X_{p,r,T}$ is a separable reflexive Banach space. Moreover, from \cite[Lemma 3.7]{JJL},
the metric space $(B_{R,T},d)$, where
\[
B_{R,T}:=\left\{{\bf u}\in {
\bf X}_{T} : \|{\bf u}\|_{{\bf X}_T}\leq R\right\},
\]
and
\[
 d({\bf u},{\bf v}):=\|u_1-v_1\|_{Y_{T}}+\|u_2-v_2\|_{Y_{T}}+\|u_3-v_3\|_{Y_{T}}
 \]
is complete.

\subsection{Uniform local existence result}
We begin with the following local existence result.
\begin{proposition} \label{pro1}
There exists $\gamma_0>0$ such that if ${\bm \psi}_0\in {\bf H}^1$ and $T\in (0,1]$ satisfy
\[
\|e^{i t\Delta}{\bm \psi}_0\|_{{\bf X}_{T}}\leq\gamma_0,
\]
then there exists a unique integral solution ${\bm \psi}(t,x)$ to \eqref{EqA1} on the time interval $[0,T]$. Moreover, $\psi_j(t,x)\in X_{p,r,T}$ for every admissible couple $(p,r)$ and satisfies the following conservation laws:
\begin{equation*}
E({\bm \psi}(t))=E({\bm \psi}_0),\\
\end{equation*}
and
\begin{equation*}
Q_1({\bm \psi}(t))=Q_1({\bm \psi}_0) \quad \hbox{ and } \quad Q_2({\bm \psi}(t))=Q_2({\bm \psi}_0).
\end{equation*}
\end{proposition}
\begin{proof}
We first prove that the existence and uniqueness in $B_{2\gamma_0,T}$ for $\gamma_0$ small enough. For any ${\bf u}=(u_1,u_2,u_3)\in {\bf X}_{T}$ and $t\in [0,T]$, we define
\begin{equation*}
\Phi(u_j)(t):=e^{it\Delta}\psi_{0,j}+i\int^{t}_{0}e^{i(t-s)\Delta}g_j({\bf u}(s))ds, \quad j\in \{1,2,3\}.
\end{equation*}
We claim that, if $\gamma_0>0$ is small enough, then $\Phi$ defines a contraction map on the metric space $(B_{2\gamma_0,T},d)$.
\vskip1mm

Let ${\bf u}\in B_{2\gamma_0,T}$, $T\in (0,1]$, and  $(\tilde{p},\tilde{r})$, $\tilde{p}=\frac{4p}{p-2}, \tilde{r}=\frac{3p}{p+1}$ two admissible pairs defined in term of a free parameter $2<p\leq6$ (to be chosen later on). By Strichartz's estimates (see \cite[Proposition 3.4]{JJL}) and H\"{o}lder's inequality, we get 
\begin{equation}\label{eqA3}
\begin{aligned}
&\|\nabla\Phi(u_1)-e^{it\Delta}\nabla\psi_1\|_{Y_{\tilde{p},\tilde{r},T}}
\leq\|\nabla g_1\left({\bf u}(s)\right)\|_{Y_{\tilde{p}',\tilde{r}',T}}
= \left(\int^{T}_{0}\|\nabla g_1\left({\bf u}(s)\right)\|^{\tilde{p}'}_{\tilde{r}'} dt\right)^{\frac{1}{\tilde{p}'}}\\
&\leq C\left(\int^{T}_{0}\|\nabla u_1\|^{5\tilde{p}'}_{\tilde{r}} dt\right)^{\frac{1}{\tilde{p}'}}+C\left(\int^{T}_{0}\|\nabla u_2\|^{2\tilde{p}'}_{\tilde{r}} dt\right)^{\frac{1}{\tilde{p}'}}+C\left(\int^{T}_{0}\|\nabla u_3\|^{2\tilde{p}'}_{\tilde{r}} dt\right)^{\frac{1}{\tilde{p}'}}\\
&\leq CT^{\frac{\tilde{p}-6}{\tilde{p}}}\|\nabla u_1\|^{5}_{Y_{\tilde{p},\tilde{r},T}}+CT^{\frac{\tilde{p}-3}{\tilde{p}}}\|\nabla u_2\|^{2}_{Y_{\tilde{p},\tilde{r},T}}+CT^{\frac{\tilde{p}-3}{\tilde{p}}}\|\nabla u_3\|^{2}_{Y_{\tilde{p},\tilde{r},T}}\leq C\gamma^2_0,
\end{aligned}
\end{equation}
provided $\gamma_0$ is small enough. Here, by the Sobolev embedding $W^{1,\tilde{r}}\hookrightarrow L^{\tilde{r}^*}$ with $\tilde{r}^*=\frac{3\tilde{r}}{3-\tilde{r}}$, we used the fact that
\begin{equation*}
\begin{aligned}
\|\nabla g_1({\bf g}(s))\|_{\tilde{r}'}&\leq C\||\nabla u_1||u_1|^{4}\|_{\tilde{r}'}+C\||\nabla (u_2u_3)|\|_{\tilde{r}'}\\
&\leq C\|\nabla u_1\|_{\tilde{r}}\|u_1\|^{4}_{\tilde{r}^*}+C\|\nabla u_2\|_{\tilde{r}}\|u_3\|_{\tilde{r}^*}+C\|\nabla u_3\|_{\tilde{r}}\|u_2\|_{\tilde{r}^*} \\
&\leq C\|\nabla u_1\|^5_{\tilde{r}}+C\|\nabla u_2\|^2_{\tilde{r}}+C\|\nabla u_3\|^2_{\tilde{r}}.
\end{aligned}
\end{equation*}

\noindent Similarly, we have
\begin{equation}\label{eqA2}
\begin{aligned}
\|\Phi(u_1)-e^{it\Delta}\psi_1\|_{Y_{\tilde{p},\tilde{r},T}}&\leq C\|u_1\|^{5}_{Y_{\tilde{p},\tilde{r},T}}+
C\|u_2\|^2_{Y_{\tilde{p},\tilde{r},T}}+C\|u_3\|^2_{Y_{\tilde{p},\tilde{r},T}}\leq C\gamma^2_0.
\end{aligned}
\end{equation}
As for  \eqref{eqA3} and \eqref{eqA2}, we have
\begin{equation*}
\|\nabla\Phi(u_2)-e^{it\Delta}\nabla\psi_2\|_{Y_{\tilde{p},\tilde{r},T}} \leq C\gamma^2_0
\end{equation*}
and 
\begin{equation*}
    \|\Phi(u_3)-e^{it\Delta}\psi_3\|_{Y_{\tilde{p},\tilde{r},T}}\leq C\gamma^2_0.
\end{equation*}
In particular, if we choose $(\tilde{p},\tilde{r})=(p_1,r_1)$ and $(\tilde{p},\tilde{r})=(p_2,r_2)$ as defined in \eqref{str-pairs}, then
\begin{equation*}
\|\Phi({\bf u})\|_{X_{T}}\leq\gamma_0+ C\gamma^2_0,
\end{equation*}
and eventually, if $\gamma_0$ is small enough in such a way that $C\gamma^2_0\leq\gamma_0$,  $B_{2\gamma_0,T}$ is an invariant set of $\Phi$.
\vskip1mm

We show now that $\Phi$ is a contraction. Let ${\bf u},{\bf v}\in B_{2\gamma_0,T}$, we have for any admissible pair $(\tilde{p},\tilde{r})$,
\begin{equation*}
\begin{aligned}
&\|\Phi(u_1)-\Phi(v_1)\|_{Y_{\tilde{p},\tilde{r},T}}=\left\| \int^{t}_{0}e^{i(t-s)\Delta}\left(g_1({\bf u}(s))-g_1({\bf v}(s))\right) ds\right\|_{Y_{\tilde{p},\tilde{r},T}}\\
&\leq C\left\|g_1({\bf u}(s))-g_1({\bf v}(s))\right\|_{Y_{\tilde{p}',\tilde{r}',T}}\\
&\leq C\left(\int^{T}_{0}\|(u_1-v_1)(|u_1|^4+|v_1|^4)+|u_2-v_2||v_3|+|u_3-v_3||u_2|\|^{\tilde{p}'}_{\tilde{r}'}\right)^{\frac{1}{\tilde{p}'}}  \\
&\leq C\big(\|\nabla u_1\|^{4}_{Y_{\tilde{p},\tilde{r},T}}+\|\nabla v_1\|^{4}_{Y_{\tilde{p},\tilde{r},T}}\big)\|u_1-v_1\|_{Y_{\tilde{p},\tilde{r},T}}+C\|\nabla v_3\|_{Y_{\tilde{p},\tilde{r},T}}\|u_2-v_2\|_{Y_{\tilde{p},\tilde{r},T}}\\
&\qquad +C\|\nabla u_2\|_{Y_{\tilde{p},\tilde{r},T}}\|u_3-v_3\|_{Y_{\tilde{p},\tilde{r},T}}.\\
\end{aligned}
\end{equation*}
If we choose $(\tilde{p},\tilde{r})=(p_1,r_1)$ and $(\tilde{p},\tilde{r})=(p_2,r_2)$, i.e., the free parameter $p=3$ and $p=6$, respectively, then
\begin{equation*}
\|\Phi(u_1)-\Phi(v_1)\|_{Y_{T}}\leq C \gamma^4_0 \|u_1-v_1\|_{Y_{T}}+C \gamma_0 \|u_2-v_2\|_{Y_{T}}+C \gamma_0 \|u_3-v_3\|_{Y_{T}}.
\end{equation*}
If $\gamma_0$ is small enough, we then have
\begin{equation*}
\|\Phi({\bf u})-\Phi({\bf v})\|_{Y_{T}}\leq C \gamma_0\left(\|u_1-v_1\|_{Y_{T}}+\|u_2-v_2\|_{Y_{T}}+\|u_3-v_3\|_{Y_{T}}\right),
\end{equation*}
and hence $\Phi$ is a contraction on $(B_{2\gamma_0,T},d)$. Therefore, $\Phi$ has a fixed point in this space.
 For the uniqueness we refer to \cite{JJL}, while for the validity of the conservation laws we cite \cite{OT}.
\end{proof}

\subsection{Orbital Stability}

Since in the context of \eqref{eqq:4.1}, we have the local existence result (Proposition \ref{pro1}),  we next prove that the set $\mathcal{G}$ is orbitally stable. Let $T^{\max}_{{\bm \psi}}$ be the maximal time of existence for the integral solution associated with \eqref{EqA1}.

\begin{lemma}\label{local-stability}
Assume that $\max\{a_1,a_2\}<D$, where $D$ is defined in \eqref{def:D}, and let ${\bf v}\in \mathcal{G}$. Then for every $\varepsilon>0$ there exists $\delta>0$ such that
\begin{equation*}
\|{\bm \psi}_0-{\bf v}\|_{{\bf H}^{1}}<\delta\Rightarrow\sup_{t\in [0,T^{\max}_{{\bm \psi}})} \emph{dist}_{{\bf H}^1}({\bm \psi}(t),\mathcal{G})<\varepsilon.
\end{equation*}
\end{lemma}
\begin{proof}
We assume by contradiction that there exist $\varepsilon_0>0$, a sequence of times $\{t_n\}\subset \R^+$, and a sequence of initial data $\{{\bm \psi}_{0,n}\} \subset {\bf H}^1$ such that the unique (for $n$ fixed) solution ${\bm \psi}_{n}(t)$ to the problem \eqref{EqA1} with initial datum ${\bm \psi}_{n}(0)={\bm \psi}_{0,n}$ satisfies
\begin{equation*}
\text{dist}_{{\bf H}^{1}}\left({\bm \psi}_{0,n},\mathcal G\right) < \frac{1}{n} \ \ \text{and} \ \ \sup_{[0,T_n]} \text{dist}_{{\bf H}^{1}}\left({\bm \psi}_{n}(t),\mathcal G\right)\ge \varepsilon_0,
\end{equation*}
where $T_n$ is the maximal time of existence of the solution arising from the initial datum ${\bm \psi}_{0,n}$.
Without loss of generality, we assume ${\bm \psi}_{0,n}\in S(a_1,a_2)$. Then by the conservation laws \eqref{eq:energy} and \eqref{eq:cons-masses}, $\{{\bm \psi}_{n}(t)\} \subset {\bf H}^1$ satisfies
\[Q_1({\bm \psi}_{n}(t))=a^2_1 \quad\text{ and } \quad  Q_2({\bm \psi}_{n}(t))=a^2_2 
\]
for any $n$, and \[
E({\bm \psi}_{n}(t)) \to m^{+}(a_1,a_2)\]
for $n\to\infty$.
From Proposition \ref{pro1}, we get the uniform local well-posedness.
By Lemmas \ref{lem2.11} and \ref{lem2.4}, as $\max\{a_1,a_2\}<D$, we  get
\begin{equation*}
\begin{aligned}
m^+(a_1,a_2)&=\inf_{{\bf u}\in S(a_1,a_2)\cap B_{\rho_0}\cap \mathcal{M}}E({\bf u})\\
&=\inf\left\{E({\bf u}) \hbox{ s.t. }{\bf u}\in S(a_1,a_2)\cap B_{\max\{a_1,a_2\}D^{-1}\rho_0}\cap\mathcal{M}\right\}.
\end{aligned}
\end{equation*}
 Note that the condition $\max\{a_1,a_2\}<D$ ensures that $\mathcal G$ is not empty.
A similar analysis to that in the proof of \cite[Theorem 1.2]{KO}
yields strict sub-additivity of $E$ on
\[
V(a_1,a_2)=S(a_1,a_2)\cap B_{\rho_0}\cap \mathcal{M}.
\]
Moreover, combining $m^+(a_1,a_2)<0$ with $E({\bm \psi}_{n}(t))\to m^+(a_1,a_2)$, we have that ${\bm \psi}_{n}(t)\in \mathcal{M}$.
Let $t_n>0$ be the first time such that
\begin{equation}\label{orb-equality}
\text{dist}_{{\bf H}^{1}}\left({\bm \psi}_{n}(t_n),\mathcal{G}\right) = \varepsilon_0.
\end{equation}
Then by the conservation laws, $\{{\bm \psi}_{n}(t_n)\} \subset B_{\rho_0}$, $Q_1({\bm \psi}_{n}(t_n))=a^2_1$ and $Q_2({\bm \psi}_{n}(t_n))=a^2_2 $ for any $n$, and $E({\bm \psi}_{n}(t_n)) \to m^+(a_1,a_2)$ as $n\to\infty$. According to the proof of Lemma \ref{lem2.6}, there exists ${\bf u}_0\in \mathcal{G}$ such that ${\bm \psi}_{n}(t_n)\to {\bf u}_0$ in ${\bf H}^{1}$. This contradicts  \eqref{orb-equality}.
\end{proof}

We now move to the proof that the solution can be extended globally in time, i.e., we show that $T^{\max}_{{\bm \psi}}=\infty$. 
\begin{proposition}\label{Pro2}
Assume that $(p,r)$ is an admissible pair with $p\neq\infty$. Then, for every $\gamma>0$ there exist $\varepsilon=\varepsilon(\gamma)>0$ and $T=T(\gamma)>0$ such that
\begin{equation*}
\sup_{\left\{{\bm \psi}\in {\bf H}^{1}\ \emph{s.t.} \ \mathop{\emph{dist}}_{{\bf H}^{1}}({\bm \psi},\mathcal{G})<\varepsilon\right\}} \|e^{it\Delta}{\bm \psi}\|_{{\bf X}_{T}}<\gamma.
\end{equation*}
\end{proposition}

\begin{proof}
We claim that, for every $\gamma>0$, there exists $T>0$ such that
\begin{equation}\label{eq:4.1bb}
\sup_{{\bm \psi}\in \mathcal{G}}\sum^{3}_{j=1}\|e^{it\Delta}\psi_j\|_{X_{p,r,T}}<\gamma.
\end{equation}
We assume by contradiction that there exist sequences ${\bm \psi}_n\subset \mathcal{G}$ and $T_n> 0$ such that $T_n\to 0$, $\gamma_0\in \R^+$ and
\begin{equation*}
\|e^{it\Delta}\psi_j\|_{X_{p,r,T}}\ge \gamma_0>0.
\end{equation*}
Since $\mathcal{G}$ is compact up to translation, up to a subsequence, there exist  $x_n\in \R^3$ and ${\bm \psi}\in H^1(\R^3,\mathbb{C}^3)$
\begin{equation*}
{\bm \psi}_n(\cdot-x_n)\to{\bm \psi}(\cdot) \ \ \text{in} \ {\bf H}^1.
\end{equation*}
By Strichartz's estimates, for every $\tilde{T}>0$, we see that, $j\in\{1,2,3\}$,
\begin{equation*}
\|e^{it\Delta}\psi_{j,n}(\cdot-x_n)\|_{X_{p,r,\tilde{T}}}\to \|e^{it\Delta}\psi_{j}(\cdot)\|_{X_{p,r,\tilde{T}}} \quad \text{as} \ n\to \infty,
\end{equation*}
and then
\begin{equation}\label{conta}
\lim_{n\to \infty}\|e^{it\Delta}\psi_{jn}(\cdot-x_n)\|_{X_{p,r,\tilde{T}}}=\lim_{n\to \infty}\|e^{it\Delta}\psi_{jn}(\cdot)\|_{X_{p,r,\tilde{T}}}=\|e^{it\Delta}\psi_{j}\|_{X_{p,r,\tilde{T}}}\geq \gamma_0.
\end{equation}
Note that, for every $\varphi\in H^1(\R^3)$, $\|e^{it\Delta}\varphi\|_{X_{p,r,T}}\leq C \|\varphi\|_{H^1}$ (see \cite[Proposition 3.4]{JJL}). By the Dominated Convergence Theorem, it follows that
\begin{equation*}
\int^{\tilde{T}}_0\|e^{it\Delta}\psi_{j}\|_{H^1}dt\to 0 \quad \text{as} \quad  \tilde{T}\to 0.
\end{equation*}
Then we have $\|e^{it\Delta}\psi_{j}\|_{X_{p,r,\tilde{T}}}\to 0$ as $\tilde{T}\to 0$. Therefore, we can choose $\tilde{T}>0$ such that \begin{equation*}
\|e^{it\Delta}\psi_{j}\|_{X_{p,r,\tilde{T}}}<\frac{\gamma_0}{2},
\end{equation*}
and this is a contradiction with respect to \eqref{conta}. Now, fix $T>0$ such that \eqref{eq:4.1bb} holds; we have
\begin{equation*}
\|e^{it\Delta}{\bm \psi}\|_{X_{p,r,T}}\leq\gamma, \quad \forall {\bm \psi}\in \mathcal{G}.
\end{equation*}
In addition, if we choose $(p,r)=(p_1,r_1)$ and $(p,r)=(p_2,r_2)$ as in \eqref{str-pairs}, we can get the estimates for the norm $\|e^{it\Delta}{\bm \psi}\|_{{\bf X}_T}$.
 \end{proof}

By Proposition \ref{pro1} and Proposition \ref{Pro2}, we have that there exists $\varepsilon_0>0$ and $T_0>0$ such that the Cauchy problem \eqref{EqA1} has a unique solution ${\bm \psi}$ on the time interval $[0,T_0]$ in the sense of Definition \ref{def1}, with  
${\bm \psi}$ satisfying $\text{dist}_{{\bf H}^1}\left({\bm \psi},\mathcal{G}\right)<\varepsilon_0$.

\begin{lemma}\label{global-T}
There exists a $\delta_0>0$ such that, provided the initial datum ${\bm \psi}_0\in {\bf H}^{1}$ satisfies $\emph{dist}_{{\bf H}^{1}}\left({\bm \psi}_0,\mathcal{G}\right)<\delta_0$, then the corresponding solution ${\bm \psi}(t)$ to \eqref{EqA1}
satisfies $T^{\max}_{{\bm \psi}}=\infty$.
\end{lemma}
\begin{proof}
Combining Proposition \ref{pro1} and Proposition \ref{Pro2}, we get that there exists $\delta_0>0$ and $T_0>0$ such that the Cauchy problem \eqref{EqA1} with associated initial datum ${\bm \psi}_0$ satisfying $\text{dist}_{{\bm H}^{1}}\left({\bm \psi}_0,\mathcal{G}\right)<\delta_0$, has a unique solution on the time interval $[0, T_0]$ in the sense of Definition \ref{def1}. Then Lemma \ref{local-stability} guarantees that the solution ${\bm \psi}(t)$ satisfies $\text{dist}_{{\bf H}^{1}}\left({\bm \psi},\mathcal{G}\right)<\varepsilon_0$ up to the maximum time of existence $T^{\max}_{{\bm \psi}}\ge T_0$.

We claim that $T^{\max}_{{\bm \psi}}=\infty$. If $T^{\max}_{{\bm \psi}}<\infty$, at any time in $(0, T^{\max}_{{\bm \psi}})$ we can get an uniform additional time of existence $T_0>0$, this contradicts the definition of $T^{\max}_{{\bm \psi}}$.
 \end{proof}
 
We can conclude with the validity of the first point in Theorem \ref{th1.2}.

\begin{proof}[Proof of Theorem \ref{th1.2} $(i)$]
The orbital stability of $\mathcal{G}$ follows from Lemmas \ref{local-stability} and \ref{global-T}.
\end{proof}

\subsection{Strong instability}
We prove now  the strong instability by blow-up, namely the second claim of Theorem \ref{th1.2}.
\begin{proof}[Proof of Theorem \ref{th1.2} $(ii)$.]
Let ${\bf v}$ be the excited state constructed in Theorem \ref{th1.1}, point \textup{(i)}. For any $s>0$, let ${\bf v}_s:=s\star{\bf v}$, and let $(\psi_{1,s},\psi_{2,s},\psi_{3,s})={\bm \psi}_s={\bm \psi}_s(t)$ be the solution to \eqref{EqA1} with the initial datum ${\bf v}_s$, defined on the interval $[0, T_{\max})$. Then, ${\bf v}_s\to {\bf v}$ as $s\to 1^+$. Moreover, it follows from \cite{BL1} that ${\bf v}\in H^1(\R^3,\R^3)$ decays exponentially at infinity, and hence $|x|{\bf v}\in L^2(\R^3,\R^3)$. It is sufficient to prove that ${\bm \psi}_s$ blows-up in finite time. Let $\sigma_{{\bf v}_s}$ be defined in Lemma \ref{lem2.3}, we have
\begin{equation*}
E({\bf v}_s)=E(s\star{\bf v})<E(\sigma_{{\bf v}_s}\star {\bf v}_s)=\inf_{\mathcal{P}^-_{a_1,a_2}}E,
\end{equation*}
because $P({\bf v}_s)<0$.

\vskip1mm

Next we infer that ${\bm \psi}_s$ blows-up in finite time. We first prove that there exists $\eta>0$ such that $P({\bm \psi}_s)\leq-\eta<0$ for any $t$ in the maximal time of existence.
Since $\sigma_{{\bf v}_s}$ is the unique global maximal point of $\Psi_{{\bf v}_s}$, the latter  is strictly decreasing and concave in $(\sigma_{{\bf v}_s},+\infty)$ (see \eqref{eq:fiber}  for the definition of $\Psi_{{\bf v}_s}$). From \cite[Section 6]{FLY},
we have the following claim: if ${\bf v}_s\in S(a_1,a_2)$ and $\sigma_{{\bf v}_s}\in (0,1)$, then
\begin{equation}\label{eq:xs5}
P({\bf v}_s)\leq E({\bf v}_s)-\inf\limits_{\mathcal{P}^-_{a_1,a_2}}E.
\end{equation}

\noindent By continuity, and $P({\bf v}_s)<0$, provided $t$ is sufficiently small we have $P({\bm \psi}_s)<0$. Therefore, from \eqref{eq:xs5},
\begin{equation}\label{P-upper-bound}
P({\bm \psi}_s)\leq E({\bm \psi}_s)-\inf\limits_{\mathcal{P}^-_{a_1,a_2}}E=E({\bf v}_s)-\inf\limits_{\mathcal{P}^-_{a_1,a_2}}E=:-\eta<0,
\end{equation}
for any $t$. Hence, we deduce by continuity that $P({\bm \psi}_s)<-\eta$ for all $t\in [0,T_{\max})$.

\vskip1mm

We introduce the virial functional 
\begin{equation}\label{eq:virdef}
I_{{\bm \psi}_s}(t):=\sum_{j=1}^3\int |x|^2|\psi_{j,s}(t)|^2  dx.
\end{equation}
By differentiating twice in time and by using \eqref{EqA1} and  \eqref{P-upper-bound}, we get 
\begin{equation}\label{eq:vir-2}
\begin{aligned}
I_{{\bm \psi}_s}^{\prime\prime}(t)&= 8 P({\bm \psi}_s(t))\leq-8\eta<0 \quad \text{for all}\quad t\in [0,T_{\max}),
\end{aligned}
\end{equation}
and using \eqref{eq:virdef} along with \eqref{eq:vir-2}, after integrating in time twice we obtain
\begin{equation*}
0\leq I_{{\bm \psi}_s}(t) \leq- 8\eta t^2+O(t) \quad \text{for all}\quad t\in [0,T_{\max}).
\end{equation*}
A convexity argument gives $T_{\max}<\infty$.
From the convexity of $I_{{\bm \psi}_s}(t)$, we derive the existence of $T_0>0$ such that 
\begin{equation}\label{eq:V1} 
\lim_{t\to T_0}\sum_{j=1}^3\int |x|^2|\psi_{j,s}(t)|^2  dx=0. \end{equation} 
Using the conservation of mixed masses \eqref{eq:cons-masses} and the Hardy inequality (see also \cite{BeFo}), we obtain
\begin{equation*}
\begin{aligned}
0<a^2_1&=Q_1({\bf v}_s)=\int |\psi_{1,s}(t)|^2+|\psi_{3,s}(t)|^2\\
&\leq2\left(\int |\nabla\psi_{1,s}(t)|^2\right)^{\frac{1}{2}}\left(\int |x|^2|\psi_{1,s}|^2\right)^{\frac{1}{2}}\\
&+2\left(\int |\nabla\psi_{3,s}(t)(t)|^2\right)^{\frac{1}{2}}\left(\int |x|^2|\psi_{3,s}(t)|^2\right)^{\frac{1}{2}},
\end{aligned}
\end{equation*}
and
\begin{equation}
\begin{aligned}
0<a^2_2&=Q_2({\bf v}_s)=\int |\psi_{2,s}(t)|^2+|\psi_{3,s}(t)|^2\\
&\leq2\left(\int |\nabla\psi_{2,s}(t)|^2\right)^{\frac{1}{2}}\left(\int |x|^2|\psi_{2,s}|^2\right)^{\frac{1}{2}}\\
&+2\left(\int |\nabla\psi_{3,s}(t)(t)|^2\right)^{\frac{1}{2}}\left(\int |x|^2|\psi_{3,s}(t)|^2\right)^{\frac{1}{2}}.
\end{aligned}
\end{equation}
Consequently, we deduce from \eqref{eq:V1} that there exists $j\in\{1,2,3\}$ such that
\[\lim_{t\to T_0}\int |\nabla \psi_{j,s}(t)|^2  dx=\infty\]
and ultimately
\begin{equation*}
\lim_{t\to T_0} \|{\bm\psi}_{s}(t)\|_{\dot{\bf H}^1}=\infty.
\end{equation*}
\end{proof}

\subsection{Absence of small data scattering}
We conclude the paper by proving the absence of small data scattering. 
\begin{proof}[Proof of Theorem \ref{th1.3}.] In order to show that small data scattering cannot hold under the assumption of Theorem \ref{th1.1} it is sufficient to prove (see also \cite{BeJe}) that our topological local minimizer ${\bf u}$ fulfills  $\displaystyle \lim_{(a_1,a_2)\to (0,0)}\|{\bf u}\|
_{\dot {\bf H}^1}=0$. By $P({\bf u})=0$, once more we write 
\begin{equation*}
E({\bf u})=\frac{1}{3}\|{\bf u}\|^2_{\dot{\bf H}^1}-\frac{3\alpha}{4} \mathrm{Re}\int u_1u_2\overline{u}_3<0
\end{equation*}
and by \eqref{eq:b7-prev} we get 
\begin{equation*}
\|{\bf u}\|_{\dot{\bf H}^1}\to 0 \quad \text{as} \quad  (a_1,a_2)\to (0^+,0^+).
\end{equation*}
 Hence, the small data scattering cannot hold.
\end{proof}

\begin{ackno}\rm
L.F. was partially supported by the  INdAM-GNAMPA Project  E53C25002010001.  X.L. was supported by the 
National Natural Science Foundation of China (Grant No. 12471103) and Anhui
Provincial Natural Science Foundation (No. 2308085MA05).   X.Y.  was supported by the National Natural Science Foundation of China (Grant No. 12401130),  the Postdoctoral Fellowship Program of CPSF (Grant No. GZC20240405) and the China Postdoctoral Science Foundation (Grant No.2024M760761).\end{ackno}

\end{document}